\documentclass[pdflatex,sn-mathphys-num]{sn-jnl}

\usepackage[T1]{fontenc}
\usepackage{graphicx}%
\usepackage{multirow}%
\usepackage{amsmath,amssymb,amsfonts}%
\usepackage{amsthm}%
\usepackage{mathrsfs}%
\usepackage[title]{appendix}%
\usepackage{xcolor}%
\usepackage{textcomp}%
\usepackage{manyfoot}%
\usepackage{booktabs}%
\usepackage{algorithm}%
\usepackage{algorithmicx}%
\usepackage{algpseudocode}%
\usepackage{listings}%

\newcommand{\uno}{ 1\!\!1}

\newcommand{\Pb}{{\bf P}_{c,\delta}}

\newcommand\RR{{\mathbb R}}
\newcommand\ZZ{{\mathbb Z}}

\numberwithin{equation}{section}
\newcommand{\supp}{\mathrm{supp}}
\newcommand{\suppF}{\supp(F)}
\newcommand{\impl}{\Rightarrow}

\makeatletter
\newtheoremstyle{plainSmallRoman}
  {3pt}{10pt}
  {\normalfont\normalfont}
  {}
  {\bfseries}
  {.}
  {.5em}
  {}
\makeatother

\theoremstyle{plainSmallRoman}

\newtheorem{theorem}{Theorem}
\newtheorem{proposition}[theorem]{Proposition}%
\newtheorem{example}{Example}%
\newtheorem{remark}{Remark}%
\newtheorem{lemma}{Lemma}
\newtheorem{corollary}{Corollary}

\newtheorem{definition}{Definition}%

\begin{document}

\title[Characterisation of distributions via record-like observations]{Characterisation of distributions via record-like observations}

\author[1]{\fnm{Ra\'ul} \sur{Gouet}}\email{rgouet@dim.uchile.cl}

\author*[2,3]{\fnm{Miguel} \sur{Lafuente}}\email{miguellb@unizar.es}

\author[2,3]{\fnm{F. Javier} \sur{L\'opez}}\email{javier.lopez@unizar.es}

\author[2,3]{\fnm{Gerardo} \sur{Sanz}}\email{gerardo.sanz@unizar.es}

\affil[1]{\orgdiv{Departamento de Ingenier\'ia Matem\'atica},
  \orgname{Universidad de Chile},
  \orgaddress{\street{Av. Beauchef 851}, \city{Santiago}, \postcode{8370456}, \country{Chile}}}

\affil*[2]{\orgdiv{Departamento de M\'etodos Estad\'isticos},
  \orgname{Universidad de Zaragoza},
  \orgaddress{\street{C/ Pedro Cerbuna, 12}, \city{Zaragoza}, \postcode{50009}, \country{Spain}}}

\affil[3]{\orgdiv{Instituto de Biocomputaci\'on y F\'isica de Sistemas Complejos (BIFI)},
  \orgname{Universidad de Zaragoza},
  \orgaddress{\street{C/ Mariano Esquillor s/n, Edificio I+D, Campus R\'io Ebro}, \city{Zaragoza}, \postcode{50018}, \country{Spain}}}

\abstract{\unboldmath We characterise probability distributions via a martingale property associated with a natural generalisation of record values, known as $\delta$-records. For an independent and identically distributed sequence $(X_n)$ with running maximum $M_n$, let $N_n$ be the number of $\delta$-records (those $X_k$ with $X_k>M_{k-1}+\delta$). We determine distributions for which $N_n-cM_n$ is a martingale, and show that this property uniquely determines the underlying distribution within broad classes.

We show that the problem can be reformulated in terms of a delay-integrated Cauchy functional equation. A distinctive feature of this equation is that it is required to hold on a set that depends on the unknown distribution itself, which both complicates the analysis and allows for a rich variety of solutions.

A complete characterisation is obtained when $\delta<0$.  For $\delta>0$, all solutions with bounded support are identified. In the case of $\delta>0$ and unbounded support, we consider both continuous and lattice distributions. In the continuous case, the characterisation reduces to a delay differential equation,  which admits classical exponential-type solutions as well as broader  families, including mixtures of exponential and gamma distributions. An analogous discrete analysis leads to difference equations whose solutions include mixtures of geometric and negative binomial distributions. In particular, this yields a new characterisation of the geometric distribution based on weak records.}



\maketitle

\section{Introduction and Preliminaries}
\label{sec:intro}

The characterisation of probability distributions is a well-established area of probability theory, with hundreds of papers and several monographs devoted to identifying properties that determine families of distributions. Beyond their role in classical extreme value theory, record values and related statistics, such as $k$-records or weak records, have been extensively used to characterise probability distributions. For example, Chapter~4 of the monograph~\cite{ABN} on record theory is devoted to distributional characterisations, while Chapter~5 of~\cite{A2017}, a monograph on general characterisations of univariate distributions, focuses specifically on records.

This record-based perspective remains an active area of research. For continuous distributions, the characterisation of symmetric distributions based on records and $k$-records is analysed in~\cite{Ahmadi2020,Ahmadi2021,GC}. For lifetime models, results in~\cite{Lee2020} show that the Weibull distribution can be characterised by the independence between a ratio of two record values and another record value, while in~\cite{KhanMustafa2024} the Rayleigh family of distributions is characterised by a property related to the conditional expectation of records.  
In the discrete case, recent work has focused mainly on the geometric distribution. This includes characterisations based on the equality in distribution between a weak record value and the sum of previous record values, as in~\cite{CastanoLopezSalamanca2013}; between a weak record value and the sum of previous (original) observations, as in~\cite{AhsanullahAliev2011}; between the sum of the first two record values and a particular linear combination of the observations, as in~\cite{AV}; or, as in~\cite{Jasinski2018,Jasinski2019}, characterisations in terms of the independence of the increments of $k$-records or weak $k$-records and the $k$-records themselves. Other discrete families, such as the discrete Pareto distribution, have also been characterised through properties of $k$-records; see~\cite{OncelAliev2016}. It is worth noting that several of these problems are closely related to the existence and uniqueness of solutions to difference, differential, or functional equations; see~\cite{Bieniek2026,BieniekMaciag2018,Shanbhag2021}.

In this paper, we study the characterisation of distributions through a natural generalisation of record values, called $\delta$-records. Given a sequence $(X_n)$ of random observations, we say that $X_n$ is a $\delta$-record if $X_n > M_{n-1} + \delta$, where $\delta$ is a real constant and $M_n=\max\{X_1,\ldots,X_n\}$; see \cite{GLS07b,GLS14,LS} for definitions, properties and applications of $\delta$-records. When $\delta=0$, $\delta$-records coincide with ordinary (upper) records. The problem we address is the following: given $\delta \ne 0$ and a sequence $(X_n)$ of independent random variables with common distribution function $F$, find all distributions $F$ such that the sequence $(N_n - cM_n)$ is a martingale (with respect to the natural filtration generated by the observations $X_n$), where $N_n$ denotes the number of $\delta$-records among the first $n$ observations. The corresponding problem when $\delta=0$, that is, when $N_n$ is the number of usual records, was completely solved in \cite{GLS07a}, through the study of an integrated Cauchy equation where the domain is the support of the distribution.

The motivation for studying this problem is related to the limit behaviour of $N_n$ as $n\to\infty$. Indeed, it was shown in~\cite{GLS07b,GLS12b} that, for geometrically distributed observations, there exists $c>0$ such that $(N_n - cM_n)$ is a martingale. This property was then used to establish normal convergence for the sequence $(N_n)$, among other asymptotic results. It was therefore natural to conjecture that this is a characteristic property of the geometric distribution among discrete models and, possibly, of the exponential distribution among continuous models.

As we shall see, our problem is equivalent to solving the following delay-integrated Cauchy functional equation:
\begin{equation*}
	1 - F(x+\delta) = c \int_x^\infty (1-F(t))\,dt, \qquad \text{for all } x \in T,
\end{equation*}
where $T$ is a certain set with $\int_T dF(x)=1$. Unlike most papers on distributional characterisations, which typically treat either the continuous or the discrete case, we address the problem for general distribution functions, although the continuous and discrete settings are studied in detail. As in~\cite{GLS07a}, the exponential and geometric distributions arise as solutions, but additional families appear as well.

Regarding our results, we completely solve the problem when $\delta<0$, providing explicit expressions for all solutions, which turn out to be discrete distributions. For lattice distributions on $\mathbb{Z}_+$, a solution exists if and only if $\delta\in[-1,0)$; in this case, the unique solution is given by the geometric distribution with parameter $c/(1+c)$. In particular, we obtain a new characterisation of the geometric distribution in terms of the number of weak records. 

When $\delta>0$, we give explicit expressions for all solutions with bounded support, which are shown to exist if and only if $c\delta<1$. The case of unbounded support is more challenging, and a complete description of the solution set appears to be out of reach. We therefore treat continuous and discrete distributions separately. In the continuous setting, the problem is reduced to the existence of positive solutions of a linear homogeneous delay differential equation (DDE). In fact, we establish a one-to-one correspondence between solutions to our problem and positive solutions of this DDE. We show that, besides particular cases of the exponential distribution, mixtures of exponential distributions, mixtures of exponential and gamma distributions, and many other families also solve the problem. The discrete case, with support on $\mathbb{Z}_+$, is treated analogously and reduced to the study of the positivity of solutions to a difference equation. Results parallel to those obtained in the continuous case are derived. In particular, mixtures of geometric distributions, as well as mixtures of geometric and negative binomial distributions, are shown to solve the problem.

The paper is organised as follows. Section~\ref{sec:intro} presents the introduction, statement of the problem, notations, definitions, and preliminary results. Section~\ref{sec:delta<0} focuses on the case $\delta<0$, while Section~\ref{sec:delta>0} analyses the case $\delta>0$. Illustrative examples are provided throughout the paper, and Appendix~\ref{sec:appendix} gathers essential technical definitions and results for completeness.

\subsection{Notation and definitions}
Let $(X_n)$ be a sequence of $\RR$-valued, integrable, independent and identically distributed (iid)  random variables, defined on a common probability space $(\Omega, \mathcal{F}, P)$, with non-degenerate distribution function $F$. The support of $F$ is defined as \begin{equation*}
	\suppF=\{x\in \mathbb{R}:	F(x+\epsilon)-F(x-\epsilon)>0,\ \forall \epsilon>0\},
\end{equation*}
 also characterised as the smallest closed set $A\subseteq\RR$, such that $F(A):=\int_AF(dx)=1$. The left  and right endpoints of $F$ are defined as $\alpha_F=\inf\{x\in\RR:F(x)>0\}=\inf\suppF$ and  $\omega_F=\sup\{x\in\RR:F(x)<1\}=\sup\suppF$, respectively. Clearly, $-\infty\le\alpha_F<\omega_F\le\infty$. Let $G=1-F$ be the survival function related to $F$.

Recall that, given an increasing family $(\mathcal{G}_n)$ of sub-sigma-algebras of $\mathcal{F}$ (also called a filtration), the random sequence $(Z_n)$ is said to be a $(\mathcal{G}_n)$-martingale if,  for all $n\ge 1$, $Z_n$ is integrable, $\mathcal{G}_n$-measurable and satisfies $E[Z_{n+1}|\mathcal{G}_n]=Z_n$.

Given $\delta\ne0$, we conventionally declare that $X_1$ is a $\delta$-record. For $n\ge2$, we say that  $X_n$ is a $\delta$-record if $X_n>M_{n-1}+\delta$, where $M_{n-1}=\max\{X_1,\ldots,X_{n-1}\}$. Let $I_1=1$ and, for $n\ge2$,  let $I_n=\uno_{\{X_n>M_{n-1}+\delta\}}$ denote the indicator variables of $\delta$-records. Finally, define $N_n=\sum_{k=1}^nI_k, n\ge1$, as the number of $\delta$-records among the first $n$ observations. 

For $n\ge1$, let $\mathcal{F}_{n}=\sigma(X_1,\ldots,X_n)$ be the $\sigma$-algebra generated by $X_1,\ldots,X_n$. Then, for $n\ge2$,
\begin{equation}\label{F1} E[I_n\;|\;\mathcal{F}_{n-1}]=
P[X_n>M_{n-1}+\delta\;|\;\mathcal{F}_{n-1}]=G(M_{n-1}+\delta),
\end{equation}
\begin{equation}\label{F2} E[M_n-M_{n-1}\;|\;\mathcal{F}_{n-1}]=
E[(X_n-M_{n-1})^+\;|\;\mathcal{F}_{n-1}]=
\int_{M_{n-1}}^\infty G(t)dt,
\end{equation}
where $a^+ = \max{\{a,0\}}$.

\begin{remark}The integrability of $X_1$ implies $\int_{x}^\infty G(t)dt<\infty$, for all $x\in\RR$ and, consequently,  \eqref{F2} is well defined. Equations \eqref{F1}, \eqref{F2} and similar expressions involving random variables, are understood in the almost sure sense.
\end{remark}

Our aim in this paper is to solve problem $\Pb$,  stated as follows: 
{\sl Given real constants $\delta\ne0$ and $c>0$, determine distribution functions $F$ on $\RR$, such that $(N_n - c M_n)$ is $(\mathcal{F}_n)$-martingale.} 

We identify $\Pb$ with its set of solutions. Specifically, if $F$ is a solution to $\Pb$, we write $F\in\Pb$ while if $\Pb$ has no solutions, we write $\Pb=\emptyset$.

The following definition and lemma are important because they allow to reformulate $\Pb$ as a functional equation.  
\begin{definition} \label{def:ST} Let $S=\{x\in\RR:H(x)=0\}$, where
	\begin{equation*}
	H(x)=G(x+\delta)-c\int_x^\infty G(t)dt, \quad  x\in\RR.
	\end{equation*}
	Let also $R=\{x\in\suppF: F(\{x\})=0, F(\{x+\delta\})>0\}$ and $T=\suppF\setminus R$.
\end{definition}

\begin{remark}\label{rem:T} $T$ is obtained by removing from $\suppF$ the set $R$ of all continuity points $x$ of $F$, such that $x+\delta$ is an atom. Clearly, $R$ contains no atoms and is countable, since it has at most as many elements as the atoms of $F$.
At first glance, this definition may seem strange, but it turns out that the martingale property depends on the behaviour of $F$ on $T$. This contrasts with the case $\delta=0$, analysed in \cite{GLS07a}, where the martingale condition was formulated as a property over the entirety of $\suppF$.
\end{remark}

\begin{lemma} 
	\label{L1} 
	Consider the statements
	\begin{enumerate}
		\item $F\in\Pb$.
		\item $F(S)=1$.
		\item $H(x)=0$, for all $x\in T$.
		\item $G(x+\delta)-G(y+\delta)=c\int_{x}^{y}G(t)dt, \,\,\text{ for all } x,y\in T$.
	\end{enumerate}
	Then the following assertions hold:
	\begin{enumerate}
		\item[(i)] $\mathrm{(a)}$, $\mathrm{(b)}$ and $\mathrm{(c)}$ are equivalent and imply $\mathrm{(d)}$.
		\item[(ii)] If $\delta>0$ then $\mathrm{(d)}$ implies $\mathrm{(c)}$.
		\item[(iii)] If $\delta<0$ and $\omega_F=\infty$ then $\mathrm{(d)}$ implies $\mathrm{(c)}$.
	\end{enumerate}
\end{lemma}
\begin{proof}
	(i) We begin by proving the equivalence between (a) and (b). Clearly, from (\ref{F1}) and (\ref{F2}), (a) is
	equivalent to $H(M_n)=0$, for $n\ge1$. In particular, (a) implies
	$H(X_1)=0$, which is equivalent to $F(S)=1$. For the converse, $F(S)=1$  means that $H(M_1)=H(X_1)=0$.  We proceed inductively noting that $H(X_n)=0$, for all $n\ge1$. If $H(M_n)=0$, then
	$H(M_{n+1})=H(X_{n+1})\uno_{\{X_{n+1}>M_{n}\}}+H(M_n)\uno_{\{X_{n+1}\le M_{n}\}}=0$, proving that $F\in\Pb$.
	
	We now check  the equivalence between (b) and (c). Suppose (b) holds; to prove (c) we consider first an atom $x\in\suppF$. Since $F(\{x\})>0$, it follows that $x\in S$; otherwise, $F(S)$ would be less than 1. On the other hand, if $x\in T$ is not an atom, then $x+\delta$ is not an atom (by the definition of $T$), so $H$ is continuous at $x$. As $x\in \suppF$, we have $F([x-1/n,x+1/n])>0$, for all $n\ge1$. Hence, there exist $x_n\in[x-1/n,x+1/n]$, such that $H(x_n)=0$, for all $n\ge1$, (since, otherwise, $F(S)<1$) and so,  by continuity of $H$ at $x$, we have that $H(x)=\lim_n H(x_n)=0$, which proves that (b) implies (c). 
	
	For the converse implication, note that (c) implies $T\subseteq S$. As $\suppF=T\cup R$,  it suffices to prove that $F(R)=0$ because this implies $F(T)=1$. This follows at once since, as commented in Remark \ref{rem:T}, $R$ contains no atoms and is countable. 
	Finally, it is clear that (d) follows from (c), by subtraction.
	
	(ii) Let $\delta>0$ and assume that (d) holds. We consider two situations depending on whether $\omega_F<\infty$ or $\omega_F=\infty$. In the first case, it is clear that $\omega_F\in T$ because $\omega_F+\delta$ is not an atom of $F$, as $\delta>0$. Let $x\in T$, such that $x<\omega_F$. Then, since $x,\omega_F\in T$, we have $G(x+\delta)-G(\omega_F+\delta)=c\int_x^{\omega_F}G(t)dt$ and so $H(x)=0$, because $G(\omega_F+\delta)=0$. 
	
	In the case $\omega_F=\infty$ we first prove that there exists a sequence $(y_n)$ in $T$ such that $y_n\uparrow\infty$.  If there is a sequence of atoms growing to $\infty$, then it can be taken as the sequence $(y_n)$. 
		On the other hand, if no such a sequence of atoms exists, then there exists $M>0$ such that $y+\delta$ is not an atom, for all $y>M$. Hence, we can take any sequence $(y_n)$ in $\suppF\cap(M,\infty)$, diverging to $\infty$, which is necessarily in $T$. Having found the sequence $(y_n)$ in $T$, note that
		\begin{equation}
		\label{eq:Gyn}
		G(x+\delta)-G(y_n+\delta)=c\int_x^{y_n}G(t)dt,
		\end{equation}
		for all $x\in T$ and $n\ge1$.  Taking limits in \eqref{eq:Gyn}, as $n\to\infty$, we get $H(x)=0$, since $G(y_n+\delta)\to0$ and $\int_x^{y_n}G(t)dt\to\int_x^{\infty}G(t)dt$, and this proves (c).
	
	(iii) Let $\delta<0$, $\omega_F=\infty$ and assume that (d) holds. Let $(y_n)$ be as in the proof of (ii), for the case $\omega_F=\infty$. Then, taking limits in \eqref{eq:Gyn}, we get  (c).
\end{proof}
\begin{remark}\label{rem:3}
 (i) It follows from Lemma \ref{L1} (i) that if $F_1, F_2\in\Pb$ have equal support, then $\lambda F_1+(1-\lambda)F_2\in\Pb$, for all  $\lambda\in[0,1]$. 
 
 (ii) Lemma \ref{L1} also implies that the ``tail-distribution'' derived from any $F\in\Pb$, is also a solution to $\Pb$. More precisely, let $x_0\in(\alpha_{F}, \omega_F)$ and define $F_0(x)=1-G(x)/G(x_0)$, if $x\ge x_0$, and $F_0(x)=0$, otherwise. Then $F_0\in\Pb$ and $\supp(F_0)=\suppF\cap[x_0,\infty)$. 
 
 (iii) Note that (c) is equivalent to the existence of a Borel set $A$, with $F(A)=1$, such that $H(x)=0$ for all $x\in A$; see Definition 1 in \cite{DG}.
\end{remark}

We show below that (c) and (d) of Lemma \ref{L1} are not equivalent. 
\begin{example}
 Let $\delta=-1, c=2$ and $F$ such that $\suppF=T=\{0,1,2\}$, with $G(0)=1/3, G(1)=1/9, G(2)=0$. It is easy to see that (d) holds but (c) does not: take first $x=0, y=1$, then $G(x+\delta)-G(y+\delta)=G(-1)-G(0)=2/3$ and $c\int_{x}^{y}G(t)dt=2\int_{0}^{1}\frac{dt}{3}=\frac{2}{3}$.
Similar calculations for $x=1, y=2$ and $x=0, y=2$ yield equalities showing that (d) holds. 
But $H(2)=G(1)-2\int_{2}^{\infty}G(t)dt=1/9$, so (c) is false.
\end{example}
Condition (c) of Lemma \ref{L1} is expressed in terms of $T\subseteq\suppF$. In contrast, for the case $\delta=0$ considered in \cite{GLS07a}, the condition was formulated directly in terms of $\suppF$.  We now present an example demonstrating that (c) is not equivalent to $H(x)=0$, for all $x\in\suppF$.

\begin{example}
	Take $\delta=1$, $c=1/5$, and  $G$ on $[0,2]$, as follows: 
\begin{equation*}
	G(x)=\begin{cases}1-\frac{x}{10},&\text{ for }x\in[0,1/2),\\
		\frac{19}{20},&\text{ for }x\in[1/2,1),\\
		\frac{19}{20}-\frac{1}{5}\left(x-1-\frac{1}{20}(x-1)^2\right),&\text{ for }x\in[1,3/2),\\
		\frac{3}{4}-\frac{1}{10}\left(x-\frac{3}{2}\right),&\text{ for }x\in[3/2,2].
	\end{cases}
\end{equation*}
Using the results in Section \ref{sec:continuous} below, it is straightforward to check that $G$ can be extended, by the method of steps, to the interval $[2,\infty)$, taking the restriction of $G$ to $[1,2]$ as initial function; this means that $H(x)=0$ holds for every $x\in[1,\infty)$.
The distribution $F=1-G$ has $\suppF=[0,1/2]\cup[1,\infty)$, while $T=[0,1/2)\cup[1,\infty)$, that is $R=\{1/2\}$. By the construction of $G$ we have $H(x)=0$, for all $x\in[0,1/2)$, implying  $F\in\Pb$. However, $0,1/2\in\suppF$,  but $G(1)-G(3/2)=\frac{1}{5}\ne\frac{1}{5}\int_0^{1/2}\left(1-\frac{x}{10}\right)dx=\frac{39}{400}$.
\end{example}

\begin{remark}
Another useful observation concerning (c) in Lemma~\ref{L1} is that the equation $H(x)=0$ can be reformulated in terms of the mean residual life (MRL) function $m(x)$,  a well-established concept in applied probability. Recall that $m(x)=E(X-x|X>x)=\int_{x}^{\infty}G(t)dt/G(x)$; hence, $H(x)=0$ is equivalent to $G(x+\delta)=c\,m(x)G(x).$ Moreover, thanks to the inversion formula $G(x)=\mu\exp\left(-\int_{0}^{x}\frac{dt}{m(t)}\right)/m(x)$, which allows to express the survival function in terms of $m$ (see \cite{Mei}), we find (omitting details) that the MRL function of $F\in\Pb$ satisfies the functional equation  
\begin{equation*}
	c\,m(x+\delta)=\exp\left(-\int_{x}^{x+\delta}\frac{dt}{m(t)}\right),\,\text{ for all }x\in T.
\end{equation*}
Under regularity conditions, the equation above can be recast as $m'(x+\delta)=m(x+\delta)/m(x)-1$, which is a delay differential equation.
\end{remark}

\begin{remark} By a direct rescaling argument, one of the parameters $c$ and $\delta$ of $\Pb$ may be fixed without loss of generality. Indeed, a simple change of variable in Lemma \ref{L1}$(c)$ shows that $F\in\Pb$ if and only if $F_a\in{\mathbf P}_{c/a,\,a\delta}$, for every $a>0$, where $F_a(x)=F(x/a)$, $x\in\mathbb{R}$.
\end{remark}

In the following corollary to Lemma \ref{L1} we prove that all  $F\in\Pb$ have finite left endpoint $\alpha_{F}$.

\begin{corollary}
	\label{cor:alphaF_finite}
If  $F\in \Pb$ then $\alpha_F>-\infty$.
\end{corollary}
\begin{proof}
	Suppose that $\alpha_F=-\infty$ and let $\epsilon>0$.  Reasoning as in the proof of Lemma \ref{L1}(ii) for the case $\omega_F=\infty$, we can find a  sequence $(x_n)$ in $T$, decreasing to $-\infty$. Clearly, this sequence can be taken such that $x_n-x_{n+1}>\epsilon$, for all $n\ge1$.  Thus, by (i) of Lemma \ref{L1}, statement (d) holds for $x_n, x_{n+1}$, that is,  
	\begin{equation}
	\label{eq:Gxn}
	G(x_{n+1}+\delta)-G(x_n+\delta)=c\int_{x_{n+1}}^{x_n}G(t)dt,\,n\ge1.
	\end{equation}
	However, since  $G(x_n), G(x_n+\delta)\to1$, as $n\to\infty$,  taking limits in  \eqref{eq:Gxn} we get the contradiction
	\begin{equation*}
	0=\lim_n c\int_{x_{n+1}}^{x_n}G(t)dt\ge  \limsup_n c\,G(x_{n})(x_n-x_{n+1})\ge c\epsilon.
	\end{equation*}
	Therefore, $\alpha_{F}>-\infty$.
	\end{proof}

In the following, we split the study into two distinct cases: $\delta<0$ and $\delta>0$. We will see that the former is easier to analyse and the complete set of distributions in $\Pb$ can be found explicitly. The case $\delta>0$ is more challenging and some questions remain open.

\section{Negative parameter $\delta$}
\label{sec:delta<0}
Throughout this section we assume $\delta<0$. In this situation we first prove that the right endpoint of distributions in $\Pb$ is  necessarily infinite and then exhibit the general solution.
\begin{proposition}
	\label{prop:omega_F infty}
	Let $F\in \Pb$, then 
	\begin{enumerate}
		\item $(\alpha_F,\alpha_F+|\delta|)\cap\suppF=\emptyset$, 
		\item $\alpha_{F}$ is an atom and 
		\item $(\alpha_F,\infty)\cap\suppF\ne\emptyset$.
	\end{enumerate}
\end{proposition}
\begin{proof}
	(a) Recall from Corollary \ref{cor:alphaF_finite} that $\alpha_F>-\infty$ and suppose there exists $x\in(\alpha_F,\alpha_F+|\delta|)\cap\suppF$; then $x\in T$ because $x+\delta\not\in\suppF$, since $\delta<0$. Also,  $\alpha_F\in T$ because the support is a closed set and $\alpha_{F}+\delta\not\in\suppF$. Then $G(x+\delta)-G(\alpha_F+\delta)=0$ but 	$\int_{x}^{\alpha_{F}}G(t)dt\ne0$. So, (d) of Lemma \ref{L1} fails and,  by (i) of the same lemma, we have a contradiction. 
	
	(b) The assertion follows at once from (a). 
	
	(c) Observe that if $(\alpha_F,\infty)\cap\suppF=\emptyset$, then $F$ would be a degenerate distribution and so, not a possible solution to $\Pb$.
\end{proof}
The general solution to $\Pb$, in the context of negative $\delta$, is presented in the following result.
\begin{theorem}
	\label{thm:delta<0}
$F\in\Pb$ if and only if the following two conditions hold: 
	\begin{enumerate}
		\item $\suppF=\{a_n:n\ge 0\}$, where $(a_n)$ is a strictly increasing sequence, with $a_0=\alpha_F$ and $a_{n+1}- a_n\ge|\delta|$, for all  $n\ge0$. In particular, $\omega_{F}=\infty$.
		\item 
		\begin{equation}\label{valoresdeltaneg}
		G(a_n)=\prod_{i=0}^{n}\frac{1}{(1+c(a_{i+1}-a_i))},\, n\ge0.
		\end{equation}
	\end{enumerate}

\end{theorem}
\begin{proof} We first prove the necessity of $(a)$ and $(b)$ for $F\in\Pb$.

$(a)$ Let $\alpha_F^0=\alpha_F$, which is an atom by (b) of Proposition \ref{prop:omega_F infty}, and define 
\begin{equation*}
\alpha_F^n=\inf\{t\in\suppF:t>\alpha_F^{n-1}\},\, n\ge1.
\end{equation*}
We use induction to prove that $\alpha_F^n$ exists and satisfies the following properties: (i) $\alpha_F^n-\alpha_F^{n-1}\ge\vert\delta\vert$; (ii) $\alpha_F^n$ is an atom; (iii)  $(\alpha_{F}^n,\alpha_{F}^n+\vert\delta\vert)\cap\suppF=\emptyset$ and (iv) $(\alpha_{F}^n,\infty)\cap\suppF\ne\emptyset$, for all $n\ge1$. 

For the case $n=1$ note that $\alpha_{F}^1$ is well defined and satisfies (i), thanks to (a) and (c) of Proposition \ref{prop:omega_F infty}. For (ii) observe that $\alpha_{F}^1\in \suppF$, because $\suppF$ is a closed set. If $\alpha_{F}^1$ is not an atom, then there exists $x,y\in\suppF$, such that $\alpha_{F}^1<x<y<\alpha_{F}^1+\epsilon$, for any $\epsilon$ such that $0<\epsilon<\vert\delta\vert$. Then $x+\delta,y+\delta\in(\alpha_{F}^0,\alpha_{F}^1)$, so $x+\delta,y+\delta\not\in\suppF$ and hence $x,y\in T$. Furthermore, $G(x+\delta)=G(y+\delta)>0$ and so, applying (d) of Lemma \ref{L1}, we get the contradiction $0=G(x+\delta)-G(y+\delta)=c\int_{x}^{y}G(t)dt>0$; therefore, (ii) holds.
 
To prove (iii), assume for contradiction that there exists $y\in(\alpha_{F}^1,\alpha_{F}^1+\vert\delta\vert)\cap\suppF$, then $y\in T$, since $y+\delta\not\in\suppF$. Moreover, $\alpha_{F}^1+\delta, y+\delta\in[\alpha_{F}^0,\alpha_{F}^1)$, which implies $G(\alpha_{F}^1+\delta)=G(y+\delta)>0$ and so, (d) of Lemma \ref{L1} yields the contradiction $0=G(\alpha_{F}^1+\delta)-G(y+\delta)=c\int_{\alpha_{F}^1}^{y}G(t)dt>0$.
Finally, for (iv) suppose that $(\alpha_{F}^1,\infty)\cap\suppF=\emptyset$. Then  $\omega_F=\alpha_F^1<\infty$ is an atom, necessarily in $T$. So, by Lemma \ref{L1} (i), $H(\omega_{F})=0$ but then, since $\delta<0$, there exists $\epsilon>0$ such that
\begin{equation*}
H(\omega_F)=G(\omega_F+\delta)-c\int_{\omega_F}^{\infty}G(t)dt=G(\omega_F+\delta)\ge G(\omega_F-\epsilon)>0.
\end{equation*}
  For general $n$ assume, as induction hypothesis, that $\alpha_{F}^k$ satisfies the properties stated above, for $k\le n$. First, the existence of 
 $\alpha_{F}^{n+1}$ is guaranteed because, by hypothesis, $(\alpha_{F}^n,\infty)\cap\suppF\ne\emptyset$. Next, the properties (i) to (iv) for $\alpha_{F}^{n+1}$, are readily checked from the induction hypothesis, by adapting the arguments for $\alpha_{F}^1$; details are omitted for brevity. 

From the construction shown above, we conclude that the sequence $(\alpha_F^n)$ is increasing and diverges to $\infty$. We also see that $\suppF$ has no points in the set $\bigcup_{n=0}^\infty(\alpha_F^n, \alpha_F^n+|\delta|)$ nor in $\bigcup_{n=0}^\infty(\alpha_F^n+|\delta|, \alpha_F^{n+1})$. Hence $\suppF$ is as claimed, with $a_n=\alpha_F^n$, for all $n\ge0$.

$(b)$ By Lemma \ref{L1} (i),
\begin{equation}
	\label{eq:recG}
	G(a_{n}+\delta)-G(a_{n+1}+\delta)=c\int_{a_{n}}^{a_{n+1}}G(t)dt=c(a_{n+1}-a_n)G(a_n), 
\end{equation}
for all $n\ge0$. But, since $a_{n+1}-a_n\ge|\delta|=-\delta$, we have $G(a_{n+1}+\delta)=G(a_{n})$ and so, for $n\ge1$, from \eqref{eq:recG} we get
$G(a_{n-1})-G(a_{n})=c(a_{n+1}-a_n)G(a_n)$, which yields
\begin{equation*}
	G(a_{n})=\frac{1}{1+c(a_{n+1}-a_n)}G(a_{n-1}), \, n\ge1.
\end{equation*}
Letting $\gamma_n=(1+c(a_{n+1}-a_n))^{-1}$ we obtain $G(a_n)=G(a_0)\prod_{i=1}^{n}\gamma_i$. 
Also,  taking $n=0$ in \eqref{eq:recG} and noting that $G(a_0+\delta)=1$ we get $$G(a_0)=\frac{1}{1+c(a_1-a_0)},$$ which implies \eqref{valoresdeltaneg}.

The sufficiency of $(a)$ and $(b)$ follows by noting that, for every sequence $\{a_n:n\ge0\}$ as in $(a)$, the function $G$  defined in \eqref{valoresdeltaneg} satisfies (d) of Lemma \ref{L1}, and so, by (iii) of that lemma, we have $F\in\Pb$.
\end{proof}

\begin{remark}Theorem \ref{thm:delta<0} completely solves the problem in the case $\delta<0$. Indeed, for any $\delta<0$ and $c>0$, all distributions $F\in\Pb$ are obtained by fixing an arbitrary sequence $(a_n)$, such that $a_{n+1}-a_n\ge\vert\delta\vert$, for all $n\ge0$, which determines  $\suppF$. The values of $G(a_n)$, for $n\ge0$, are given by \eqref{valoresdeltaneg}. In particular, all distributions in $\Pb$ are discrete.
\end{remark}

We end this section with an example involving a lattice distribution.
\begin{example} (Distributions on $\ZZ^+$) Let $F\in\Pb$ have $\suppF=\ZZ^+=\{0,1,2,\ldots\}$, then, by Theorem \ref{thm:delta<0}, $\delta\in[-1,0)$. The sequence $(a_n)$, with $a_n=n$, for $n\ge0$, satisfies the condition in Theorem \ref{thm:delta<0} (a). Hence, from \eqref{valoresdeltaneg} we get $G(n)=(1+c)^{-n-1}$, for $n\ge0$; that is, $F$ is the geometric distribution starting at 0, with parameter $c/(1+c)$.
Note that, when $\delta\in[-1,0)$, $\delta$-records are just weak records, as defined in \cite{Ver}, since $X_n>M_{n-1}+\delta$ if and only if $X_n\ge M_{n-1}$. For characterisation of distributions through weak records, see, for instance \cite{AV,WL}. We point out the similarity with the case $\delta=0$ (ordinary records) studied in \cite{GLS07a}. In that case, there are no distributions in ${\bf P}_{c,0}$ with support on $\ZZ^+$ if $c \ge1$. For $c<1$, the solutions are mixtures of a Dirac mass at 0 and a geometric distribution starting at 1, with parameter $c$.
\end{example}	
\section{Positive parameter $\delta$}
\label{sec:delta>0}
In this section we study problem $\Pb$ under the assumption $\delta>0$. We first analyse the case of distributions with finite right endpoint and present a complete solution. The proofs of results for this case are analogous to those of Proposition \ref{prop:omega_F infty} and Theorem \ref{thm:delta<0}.
\begin{proposition}
\label{prop:delposfinito}
Let  $F\in\Pb$, such that $\omega_{F}<\infty$, then
\begin{enumerate}
	\item $(\omega_{F}-\delta,\omega_{F})\cap\suppF=\emptyset$,
 	\item $\omega_{F}$ is an atom.
\end{enumerate}
\end{proposition}
\begin{proof}
	(a) Note first that $\omega_{F}\in T$ because, as the support is closed,  $\omega_{F}\in \suppF$ and $\omega_F+\delta\not\in\suppF$. Also, if $x\in(\omega_{F}-\delta,\omega_{F})\cap\suppF$, then $x+\delta\not\in\suppF$, so $x\in T$. We have $G(x+\delta)-G(\omega_{F}+\delta)=0$ but $\int_{x}^{\omega_F}G(t)dt>0$, so (d) of Lemma \ref{L1} fails and we get a contradiction. Therefore, assertion (a) holds. (b) This result is direct from (a).
\end{proof}
\begin{theorem}
	\label{thm:delposfinito}
Let $F$ be such that $\omega_F<\infty$. Then $F\in\Pb$ if and only if the following two conditions hold:
\begin{enumerate}\item $\suppF=\{a_0,a_1,\ldots,a_m\}$ where $m\ge1, a_0=\alpha_F$, $\delta<a_{n+1}-a_{n}<1/c$, for $n=0,1,\ldots,m-2$; $\delta<a_{m}-a_{m-1}=1/c$ and $a_m=\omega_F$; 
\item $G(a_0)\in(0,1)$ and	\begin{equation*}
	G(a_n)=G(a_0)\prod_{i=0}^{n-1}(1-c(a_{i+1}-a_i)),\, \text{for } n=1,\ldots,m-1.
	\end{equation*} 
\end{enumerate}
\end{theorem}

\begin{proof}We first prove the necessity of $(a)$ and $(b)$ for $F\in\Pb$.

As in the proof of Theorem \ref{thm:delta<0}, we sequentially define the elements of $\suppF$ but in reverse order, starting from $\omega_F$. Let $\omega_F^0=\omega_F$ and 
$\omega_F^1 = \sup\{t \in \suppF : t < \omega_F^0\}$, which is well defined, because $F$ is non-degenerate. 

The following properties of $\omega_F^1$ are of interest: (i) $\omega_F^1\le \omega_F^0 - \delta$, which follows directly from (a) in Proposition \ref{prop:delposfinito}; (ii) $(\omega_F^1,\omega_{F}^0)\cap\suppF=\emptyset$, a simple consequence of the definition; (iii) $\omega_{F}^1$ is an atom, which we demonstrate by contradiction: If $\omega_{F}^1$ is not an atom, there exist $x,y\in\suppF$ such that $\omega_F^1-\epsilon< x<y<\omega_F^1$, for any small enough $\epsilon>0$. Then $x+\delta, y+\delta\in(\omega_F^1,\omega_F^0)$ and so, by property (ii),  $x+\delta, y+\delta\not\in\suppF$. This implies that $x,y\in T$ and then, noting that $G(x+\delta)=G(y+\delta)$, from (d) in Lemma \ref{L1}, we get the contradiction $0=G(x+\delta)-G(y+\delta)=c\int_{x}^{y}G(t)dt>0$; 
(iv) $\omega_F^1<\omega_F^0 - \delta$. To prove this property (the strict version of (i)) we show that $\omega_F^0-\delta$ is not an atom and hence, thanks to (iii), equality in (i) is impossible. Suppose, on the contrary, that $\omega_F^0-\delta$ is an atom. Then $\omega_F^0-\delta\in T$ and noting also that $\omega_F^0\in T$ by (b) in Proposition \ref{prop:delposfinito}, Lemma \ref{L1} (d) yields the contradiction $0=G(\omega_F^0-\delta+\delta)-G(\omega_F^0+\delta)=c\int_{\omega_F^0-\delta}^{\omega_F^0}G(t)dt>0$; (v) $L_1:=(\omega_F^1-\delta,\omega_F^1)\cap\suppF=\emptyset$. Reasoning by contradiction,  if $x\in L_1$ then $x+\delta\in(\omega_F^1,\omega_{F}^0)$ and, by property (ii), $x+\delta\not\in\suppF$, hence $x\in T$. Also, as $\omega_F^1\in T$ because of (iii), we apply (d) of Lemma \ref{L1} to reach the contradiction $0=G(x+\delta)-G(\omega_F^1+\delta)=c\int_{x}^{\omega_F^1}G(t)dt>0$.
		
Now, if $\alpha_F > \omega_F^1 - \delta$, then the only possibility is $\omega_F^1 = \alpha_F$, by (v) and because $\alpha_F \in \suppF$. So, $m = 1$, $a_0 = \alpha_F$, and $a_1 = \omega_F$. This situation implies the strict inequality $\delta < \omega_F - \alpha_F$, by (iv). Furthermore, invoking Lemma \ref{L1} (c) again, we get $c = (\omega_F - \alpha_F)^{-1}$, which yields $c\delta < 1$. Hence, the theorem is proved in the case just described.

If $\alpha_F \leq \omega_F^1 - \delta$, then $m \geq 2$, and, as above, we define $\omega_F^2 = \sup\{t \in \suppF : t < \omega_F^1\}$, which exists because $\alpha_F \in \suppF$. Moreover, analogous arguments can be used to verify that the properties of interest (i)--(v), stated above, also hold when $\omega_F^0, \omega_F^1$ are replaced by $\omega_F^1, \omega_F^2$, respectively; details are omitted for brevity.
		
By iterating the procedure that defines $\omega_F^1$ and $\omega_F^2$, we obtain a finite decreasing sequence $\omega_F^0 > \omega_F^1 > \cdots > \omega_F^m = \alpha_F$, such that $\omega_F^{n-1} - \omega_F^n > \delta$, for $n=1, \ldots, m$, with $m=\max\{n\ge0: \alpha_F\le\omega_F^n-\delta\}+1$. Then $\suppF = \{a_0, \ldots, a_m\}$, with $a_n = \omega_F^{m-n}$, for $n=0, \ldots, m$. 
		
The construction above shows that $a_{n+1}-a_{n}>\delta,$ for $n=0, \ldots, m-1$, which implies $G(a_n+\delta)=G(a_n),$ for $n=0, \ldots, m$. So, by Lemma \ref{L1} (d) we have
\begin{equation*}
			G(a_n)-G(a_{n+1})=G(a_n+\delta)-G(a_{n+1}+\delta)=c\int_{a_n}^{a_{n+1}}G(t)dt=cG(a_n)(a_{n+1}-a_n), 
\end{equation*}
		which yields
		\begin{equation}
			\label{recposfin}G(a_{n+1})=(1-c(a_{n+1}-a_n))G(a_n),
		\end{equation}
		for $n=0, \ldots, m-1$. Now, as $G(a_n)>0$, it follows from \eqref{recposfin} that $a_{n+1}-a_n<1/c$, for  $n=0, \ldots, m-2$,  as claimed. Moreover, as $G(a_m)=G(\omega_F)=0$, \eqref{recposfin} yields $a_{m}-a_{m-1}=1/c$.

For the sufficiency of $(a)$ and $(b)$ note that, for every set $\{a_0,\ldots,a_m\}$ as in $(a)$, the values $G(a_n)$ as in $(b)$ define a solution in $\Pb$ with $\supp(F)=\{a_0,\ldots,a_m\}$, by  Lemma \ref{L1} (ii).
\end{proof}

The following result is an immediate consequence of Theorem \ref{thm:delposfinito}.
\begin{corollary} There exists $F\in\Pb$ with $\omega_F<\infty$ if and only if $c\delta<1$.\end{corollary}

\begin{remark} Note that, for distributions with support as described in Theorem \ref{thm:delposfinito}, $\delta$-records are equivalent to standard records because $X_n>M_{n-1}+\delta$ if and only if $X_n>M_{n-1}$. Therefore, the solutions found here coincide with those of Example 3.4 in \cite{GLS07a}.
\end{remark}

Having addressed the case where $\omega_F<\infty$, we now turn to the more challenging scenario of $\omega_F=\infty$. We analyse the continuous and discrete (lattice) distributions separately.
\subsection{Continuous distributions}
\label{sec:continuous}
We consider here the existence of continuous distributions $F\in\Pb$ which, because of Theorem \ref{thm:delposfinito}, necessarily have $\omega_F=\infty$. Recall from Corollary \ref{cor:alphaF_finite} that $\alpha_F>-\infty$. So, without loss of generality, we can restrict attention to distributions $F$ with support  in $\RR_+$ and $\alpha_{F}=0$. 
\begin{example}
\label{exponential}
Let $F$ be the exponential distribution with parameter $\theta>0$ (denoted $Exp(\theta)$). That is, $F(x)=0$, for $x<0$ and $F(x)=1-e^{-\theta x}$, for $x\ge0$. Note that  $F$ being continuous implies $T=\suppF=\RR_+$ (recall Definition \ref{def:ST}). So, from Lemma \ref{L1} (i), we have $F\in\Pb$ if and only if
\begin{equation*}
	e^{-\theta(x+\delta)}=c\int_{x}^{\infty}e^{-\theta t}dt=\frac{c}{\theta}e^{-\theta x}, \, x\ge0,
\end{equation*}
which is equivalent to $\theta e^{-\theta \delta}=c$. In other words, the $Exp(\theta)$ distribution solves $\Pb$ if and only if $\theta$ is a solution to the equation $\theta e^{-\theta \delta}=c$. It is easy to check that such a solution exists if and only if $c\delta\le e^{-1}$ and, when $c\delta< e^{-1}$, there are two of them. See Example \ref{gamma} for further discussion and extensions.
\end{example}
Motivated by the preceding example we look for continuous solutions $F$ with $\suppF=\RR_+$. This is still a  restricted class of continuous distributions not including, for instance, those with singular components (in Lebesgue's decomposition) or having ``flat'' segments.

Throughout the remainder of this subsection, we assume that distributions $F$ under consideration are continuous, with $\suppF=\RR_+$. It then follows from Lemma \ref{L1} (i) that $F\in\Pb$ if and only if 
$G(x)=c\int_{x-\delta}^{\infty}G(t)dt$, for all $x\ge\delta$. This implies that $F$ is absolutely continuous on $(\delta,\infty)$ and satisfies the homogeneous delay differential equation (DDE)
$y'(t)+cy(t-\delta)=0, t\ge\delta$. The solution to this DDE is obtained by specifying an initial function $\varphi$ and solving the  initial value problem
\begin{equation}
	\label{eq:DDE}
	y'(t)+cy(t-\delta)=0,\,  t\ge\delta;\quad y(t)=\varphi(t),\,   t\in[0,\delta].
\end{equation} 
In our context, as we look for continuous solutions with $\suppF=\RR_+$, $1-\varphi$ should be the initial segment of a continuous and strictly increasing distribution function on $\RR_+$. So $\varphi$ must be continuous and strictly decreasing on $[0,\delta]$, with $\varphi(0)=1$ and, because $\suppF=\RR_+$, we also need $\varphi(\delta)>0$. Let the set of such functions be denoted by $\Phi$. Observe that, in general, $\varphi$ only needs to be integrable for \eqref{eq:DDE} to have a solution; see Theorem 2.1 in \cite{HL}. Note also that, if $y$ satisfies \eqref{eq:DDE}, then
\begin{equation}
\label{eq:DDE1}
y(s)-y(t)=c\int_{s-\delta}^{t-\delta}y(x)dx, \,\ s,t\ge\delta.
\end{equation}
A solution $y$ to \eqref{eq:DDE} yields a solution to problem $\Pb$ if and only if $1-y$  is a continuous distribution on $\RR_+$. That is, $y$ should be positive, continuous and strictly decreasing on $\RR_+$, with $y(0)=1$ and $\lim_{t\to\infty}y(t)=0$.

The (unique) solution to \eqref{eq:DDE} on $[\delta,\infty)$ is computed as a continuous extension of $\varphi$, by applying the so-called method of steps, which uses \eqref{eq:DDE1}; see \cite{HL}. For example, the first step yields 
\begin{equation}
	\label{eq:step1}
y(t)=\varphi(\delta)-c\int_{0}^{t-\delta}\varphi(s)ds, \,t\in[\delta, 2\delta].
\end{equation}
In general, for $k\ge1$, we have 
\begin{equation}
	\label{eq:stepk}
y(t)=y(k\delta)-c\int_{(k-1)\delta}^{t-\delta}y(s)ds, \,t\in[k\delta, (k+1)\delta].
\end{equation}
Hence, any  $F\in\Pb$ is completely determined by its behaviour on the initial interval $[0, \delta]$. Moreover, $F(t)$ is increasingly smooth as $t$ increases, with derivatives of order $k$ on  $(k\delta,\infty)$ for $k \geq 1$.
We summarise our findings in the next theorem.
\begin{theorem} 
	\label{thm:bijection}
	 There exists a bijection between the set of continuous distributions $F\in\Pb$, with $\suppF=[0,\infty)$, and the set of positive solutions $y$ to \eqref{eq:DDE}, with  initial function $\varphi\in\Phi$. The bijection is given by $F=1-y$. In particular, there exists $F\in\Pb$ with $\supp(F)=[0,\infty)$ if and only if $c\delta\le1/e$.
\end{theorem}
\begin{proof}
From the preceding discussion, it is clear that every $F\in\Pb$, with $\suppF=[0,\infty)$, is a solution to \eqref{eq:DDE}, with initial function $\varphi\in\Phi$, given by $\varphi(t)=1-F(t)$ for $t\in[0,\delta]$. Conversely, any positive solution $y$ to \eqref{eq:DDE}, with initial function $\varphi\in\Phi$, is  continuous, strictly decreasing and, by Lemma \ref{lem:expdecay} (ii), $y(t)\to0$ as $t\to\infty$. Hence, $1-y$ is a continuous distribution function in $\Pb$, with $\suppF=[0,\infty)$. Finally, by Theorem \ref{thm:deltac<1/e}, if $c\delta>1/e$, there is no positive solution to \eqref{eq:DDE}, so $\Pb=\emptyset$.
\end{proof}
The result above fully describes the elements of $\Pb$ in terms of their values on the interval $[0,\delta]$, that is, in terms of the initial function $\varphi$.  
Then, an interesting problem is to find conditions on initial functions $\varphi\in\Phi$ that generate solutions to $\Pb$. We address this in the following sections, providing both necessary and sufficient conditions on  $\varphi$  for the positivity of solutions to \eqref{eq:DDE}. From this point onward, we assume that $c\delta\leq 1/e$ and define the parameter $a=c\delta/2$.

\subsubsection{Necessary condition}
\begin{proposition}
	\label{prop:nec}
	If $y$ is a positive solution to problem \eqref{eq:DDE}, with initial function $\varphi\in\Phi$, then
\begin{equation}
		\label{eq:nec2}
		\varphi(\delta)>\frac{2a(I_1-2a I_2)}{1-2a},
		\end{equation}
	where $I_1=\frac{1}{\delta}\int_{0}^{\delta}\varphi(t)dt$ and $I_2=\frac{1}{\delta^2}\int_{0}^{\delta}\int_{0}^{t}\varphi(s)dsdt$.
\end{proposition}
\begin{proof}
	From \eqref{eq:stepk}, with $k=1,2$,  we have
	\begin{equation*}
	\begin{split}
	y(3\delta)&=y(2\delta)-c\int_{\delta}^{2\delta}\left(\varphi(\delta)-c\int_{0}^{t-\delta}\varphi(s)ds\right)dt\\
	&=y(2\delta)-c\delta\varphi(\delta)+c^2\int_{\delta}^{2\delta}\int_{0}^{t-\delta}\varphi(s)dsdt\\
	&=y(2\delta)-2a\left(\varphi(\delta)-2a I_2\right)\\
	&=(1-2a)\varphi(\delta)-2a(I_1-2a I_2).
	\end{split}
	\end{equation*}
	Thus, the (necessary) positivity of $y(3\delta)$ is equivalent to \eqref{eq:nec2}.
\end{proof}
\begin{remark} We could continue applying the method of steps to obtain sharper bounds on $\varphi(\delta)$, but the expressions become unwieldy.
\end{remark}

\subsubsection{Sufficient conditions}
We consider here the derivation of sufficient conditions on $\varphi$ to yield a positive solution. 
The idea is to construct a sequence $(a_n)$ that serves, in some sense, as a lower bound for $y$ and to find conditions for the positivity of $(a_n)$. 
We begin with  a technical lemma. Other results used in the proof of Proposition \ref{prop:suf}, below, can be found in the Appendix.
\begin{lemma}
	\label{lem:bn1}
	Let $y$ be the solution to \eqref{eq:DDE}, with initial function $\varphi\in\Phi$. Define  $b_k=y((k+1)\delta)$, for $k\ge0$, and let $n\ge1$. If $b_k>0$, for $k=1,\ldots,n$, then $y$ is positive, decreasing and convex on $[k\delta,(k+1)\delta]$, for $k=1,\ldots,n$. Moreover, 
	\begin{equation*}
	b_{k+1}\ge (1-a)b_{k}-ab_{k-1},
	\end{equation*}
	for $k=1,\ldots,n$.
\end{lemma}
\begin{proof}
	For simplicity, let $J_k=[k\delta, (k+1)\delta]$, for $k\ge0$, and define the following statements depending on $n$: $p(n)=``b_k>0$, for $k=1,\ldots,n$'';  $q(n)=``y$ is positive, decreasing and convex on $J_k$, for $k=1,\ldots,n$'' and $r(n)=``b_{k+1}\ge (1-a)b_{k}-ab_{k-1}$, for $k=1,\ldots,n$''. We must then prove that $p(n)\impl q(n)\wedge r(n)$, for all $n\ge1$. To that end it suffices to establish $p(n)\impl q(n)$ and $q(n)\impl r(n)$, for all $n\ge1$.
	
	For the first implication we use induction on $n$. In the initial step   assume $b_1>0$ and observe that \eqref{eq:step1} implies that $y$ decreases and $y'$ increases (so $y$ is convex) on $J_1$, because $\varphi$ is positive and decreases on $J_0$. We have thus proved that $p(1)\impl q(1)$. Now, as  induction hypothesis, suppose $p(n)\impl q(n)$ and assume that $p(n+1)$ holds. As $p(n+1)$ implies $p(n)$, we have $q(n)$ which, by \eqref{eq:stepk} with $k=n+1$,  yields $q(n+1)$, and the induction is complete.
	
	For the second implication suppose $q(n)$ holds. As $y$ is convex on $J_k$, the following bound holds, for $k=1,\ldots,n$:
		\begin{equation*}
		\int_{k\delta}^{(k+1)\delta}y(s)ds\le\frac{\delta}{2}(y(k\delta)+y((k+1)\delta)).
	\end{equation*}
	The inequality above and \eqref{eq:stepk}, with $k$ replaced by $k+1$ and $t=(k+2)\delta$, yield
	\begin{equation*}
b_{k+1}\!=y((k+1)\delta)-c	\int_{k\delta}^{(k+1)\delta}y(s)ds\ge (1-a)y((k+1)\delta)-ay(k\delta)=(1-a)b_k-ab_{k-1},
	\end{equation*}

	for $k=1,\ldots,n$, and the proof is complete.
	\end{proof}
\begin{proposition}
	\label{prop:suf}
	Suppose $a<3-2\sqrt{2}$ and let $y$ be the solution to \eqref{eq:DDE}, with initial function $\varphi\in\Phi$. Then $y$ is positive if 
	\begin{equation}
	\label{eq:suf}
	\varphi(\delta)>\beta:=\frac{2a((1-\lambda_2)I_1-2aI_2)}{1-\lambda_2-2a},
	\end{equation} 
	where $\lambda_2=\tfrac{1}{2}(1-a-\sqrt{D})$, $D=(1-a)^2-4a$ and $I_1, I_2$ are defined in Proposition \ref{prop:nec}.
\end{proposition}
\begin{proof}
We use the notations and definitions of  Lemma \ref{lem:bn1}. Note first that $\beta$ is well-defined, since  $1-\lambda_2-2a>(1-3a)/2>(1-3(3-2\sqrt{2}))/2>0$, by hypothesis. Moreover, as $I_1>I_2$, we have $\beta>2aI_1>0$. 

By Lemma \ref{lem:bn1}, it suffices to prove that $b_n>0$, for all $n\ge1$. We proceed by induction noting initially, from \eqref{eq:step1}, that  $b_1=\varphi(\delta)-2aI_1>\beta-2aI_1>0$. As the induction hypothesis assume that $p(n)$ holds. Then, from Lemma \ref{lem:bn1}, we get  $r(n)$.

Now we invoke Lemma \ref{lem:bn2} with $x_k=b_{k+1}$ for $k=0,\ldots,n$. Hence, we have $b_{k+1}=x_k \geq a_k>0$ for $k=2, \ldots, n$, where $(a_n)$ is the solution of the recurrence \eqref{eq:rec} in Lemma \ref{lem:rec}, with $a_0 = b_1$ and $a_1 = b_2$. Note that the conditions of Lemma \ref{lem:rec} are satisfied because hypothesis \eqref{eq:suf} is equivalent to $b_2 > \lambda_2 b_1$. From the arguments above, we conclude that $b_{n+1} > 0$, and the induction is complete.
\end{proof}
\begin{corollary}
	\label{cor:sufR} Under the hypotheses of Proposition \ref{prop:suf}, $y$ is positive if  
		\begin{equation}\label{eq:suf2}
		\varphi(\delta)>\frac{2a}{1-\lambda_2}.
		\end{equation}
\end{corollary}
\begin{proof}
It suffices to prove that $\beta\le \frac{2a}{1-\lambda_2}$, for every $\varphi\in\Phi$, where $\beta$ is defined in \eqref{eq:suf}. To that end, we define the function $g(s)=\frac{1}{\delta}\int_0^s\varphi(u)du$, for $s\in[0,\delta]$, which is continuous, increasing and concave, with $g(0)=0$ and $g(\delta)=I_1$. Maximising $\beta$ as a function of $\varphi$ is equivalent to maximizing $\gamma g(\delta)-\frac{1}{\delta}\int_0^\delta g(s)ds$ as a function of $g$, where $\gamma=(1-\lambda_2)/(2a)>0$. As $g$ is increasing and concave, we have $g(t)\ge tg(\delta)/\delta$ on $[0,\delta]$ and so, $\int_0^\delta g(s)ds\ge  \delta g(\delta)/2$. Hence, 
\begin{equation*}
\beta=\frac{4a^2\left(\gamma g(\delta)-\frac{1}{\delta}\int_0^\delta g(s)ds\right)}{1-\lambda_2-2a}
\le\frac{4a^2(\gamma-1/2)g(\delta)}{1-\lambda_2-2a}
\le\frac{2a(1-\lambda_2-a)}{1-\lambda_2-2a}=\frac{2a}{1-\lambda_2},
\end{equation*}
where the last equality above follows from elementary calculations.
\end{proof}

\begin{remark} (i) Observe the difference between the sufficient conditions in Proposition \ref{prop:suf} and Corollary \ref{cor:sufR}. While the bound of $\varphi(\delta)$ in \eqref{eq:suf} depends on $\varphi$, the one in \eqref{eq:suf2} does not.

(ii) Proposition \ref{prop:nec} provides the necessary condition  \eqref{eq:nec2} on the initial function $\varphi$ for \eqref{eq:DDE} to have a positive solution. On the other hand, Proposition \ref{prop:suf} gives the sufficient condition \eqref{eq:suf} but imposes the constraint $a<3-2\sqrt{2}\sim 0.1716$. Notably, this bound is quite close to the condition $a\le1/2e\sim 0.1839$, which is necessary for the existence of a positive solution.
\end{remark}

\subsubsection{Comparison of necessary and sufficient bounds}
We are interested in whether the necessary and sufficient conditions on $\varphi$ differ significantly. Conditions \eqref{eq:nec2} and \eqref{eq:suf} for $\varphi(\delta)$ can be equivalently written, respectively as 
\begin{equation*}
\begin{split}
\varphi(\delta)/I_1>N&:=2a(1-2a I_2/I_1)/(1-2a),\\
\varphi(\delta)/I_1>S&:=2a((1-\lambda_2)-2a I_2/I_1)/(1-2a-\lambda_2).
\end{split}
\end{equation*}
 Note that $N$ and $S$ depend on $I_1$ and $I_2$ through their ratio $r:=I_2/I_1$, which satisfies $r\in[1/2,1]$. In Figure \ref{fig:NS} we exhibit plots of $N$ and $S$, as functions of $a$, for $\delta=1$ and $r=0.55, 0.75, 0.95$, to assess how close the conditions are. The graphs illustrate that the conditions are quite similar.

\begin{figure}[h]

	\centering
	\includegraphics[width=12.1cm]{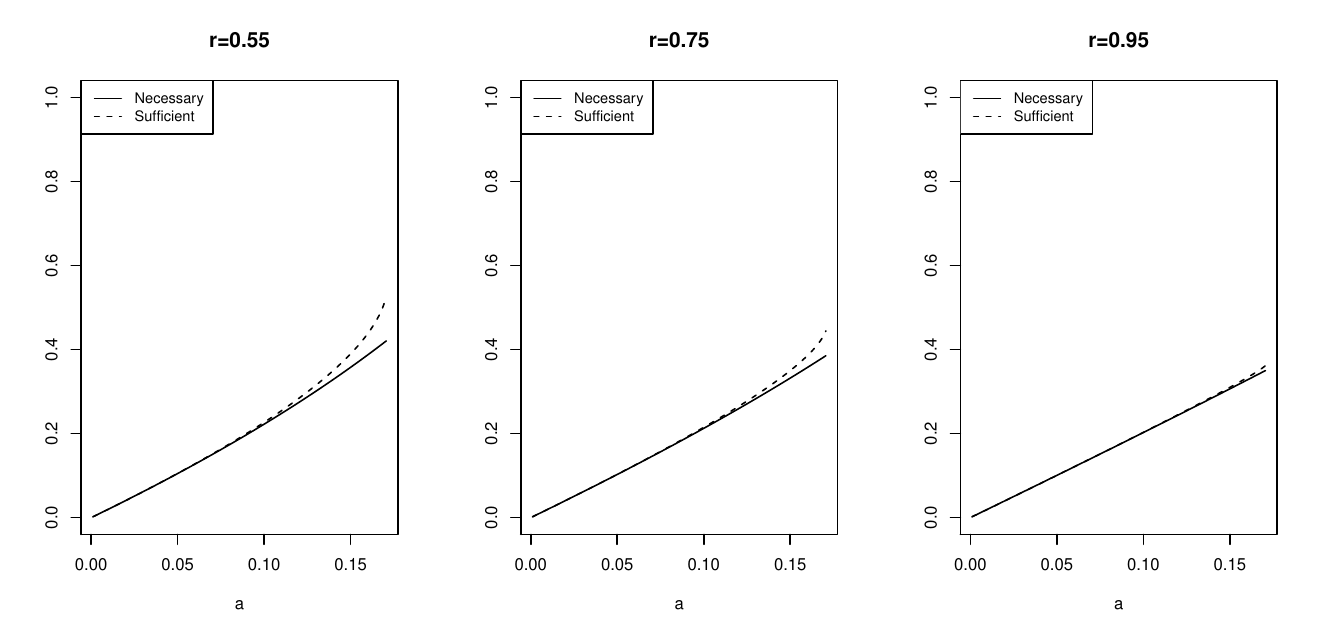}\vspace{-5pt}
	\caption{	\label{fig:NS}\footnotesize In each panel, the region above the dashed curve corresponds to pairs $(a,\varphi(\delta)/I_1)$, such that $\varphi$ yields a positive solution to \eqref{eq:DDE}.  Points below the solid curve corresponds to pairs such that $\varphi$ does not yield a positive solution.}
\end{figure}
\subsubsection{Comparison of solutions}
We have seen above that the necessary and the sufficient conditions on the initial function are close. However, certain cases remain unresolved, where for a given $\varphi\in\Phi$, it is unclear whether it generates a solution to $\Pb$. One such case occurs when $a\in[3-2\sqrt{2},1/(2e)]$, where necessary conditions are known, but no sufficient condition has been established. Nevertheless, solutions do exist in this range, including specific exponential distributions (see Examples \ref{exponential} and \ref{gamma}). Corollary \ref{cor:compare} below establishes sufficient conditions by comparing the given solution with another. This result is based on a general comparison theorem for DDEs presented in the Appendix.
\begin{corollary}\label{cor:compare}
	Let $F\in\Pb$. If $y$ is the solution to \eqref{eq:DDE}, with initial function $\varphi\in\Phi$, such that $\varphi(t)\le G(t)$ for $t\in[0,\delta]$ and $\varphi(\delta)=G(\delta)$, then $y(t)\ge G(t)$, for all $t\ge\delta$ and $1-y\in\Pb$.
\end{corollary}
\begin{proof}
	The assertion follows directly from Theorems \ref{thm:bijection} and \ref{thm:compare}.
\end{proof}
\begin{example}\label{gamma}
In Example \ref{exponential}, we found that for $a<1/2e$ (recall that $a=\delta c/2$), the $Exp(\theta_1)$ and $Exp(\theta_2)$ distributions are in $\Pb$, where $\theta_1, \theta_2$ are the only real solutions to $\theta e^{-\theta\delta}=c$. Here, we analyse the case $a=1/2e$, where only one real solution to the equation exists, namely $\theta=1/\delta$. Thus, the $Exp(1/\delta)$ distribution belongs to $\Pb$. Interestingly, there are also solutions of the form $G(t)=(\alpha t+1)e^{-t/\delta}$, for some $\alpha>0$. Indeed, it is straightforward to check that $G'(t)=-G(t-\delta)/(\delta e)$, for all $t\ge\delta$, and $G(0)=1$. It only remains to verify that $G$ is decreasing, which holds if and only if $\alpha\delta\le1$. So, the distribution defined by $F(t)=1-(\alpha t+1)e^{-t/\delta}$, $t>0$, is in $\Pb$ for every $\alpha\in[0,1/\delta]$ when $a=1/2e$. Since the survival function of the Gamma distribution with shape parameter $p=2$ and rate parameter $\theta=1/\delta$ (denoted $Gamma(2,1/\delta)$) is $G_1(t)=(t/\delta+1)e^{-t/\delta}$, $t>0$, our solution corresponds to the mixture of the $Gamma(2,1/\delta)$ and $Exp(1/\delta)$ distributions, with respective weights $\alpha\delta$ and $1-\alpha\delta$.
	\end{example}

\begin{remark} Example \ref{gamma}, combined with Corollary \ref{cor:compare}, allows us to identify new solutions to $\Pb$ through comparison. Notably, the case $a=1/2e$ is not covered by Proposition \ref{prop:suf}, which requires $a<3-2\sqrt{2}$.  Specifically, if  $G(t)=e^{-t/\delta}$, for $t\ge0$, Corollary \ref{cor:compare} implies that any initial function $\varphi\in\Phi$, satisfying $\varphi(\delta)=G(\delta)=1/e$ and $\varphi(t)\le e^{-t/\delta}$, for $t\in[0,\delta)$, generates a solution in $\Pb$.
\end{remark}

Another direct application of Theorem \ref{thm:compare} allows for the description of a large family of solutions to $\Pb$. This result is obtained through a direct comparison with the fundamental function of \eqref{eq:DDE}, which is known to be positive. See Definition \ref{def:fundamental} and Theorem \ref{thm:fundamental>0}.

\begin{proposition}
		\label{prop:initial}
	Let $\psi$ be a positive and integrable function on $[0,\delta]$, with $\int_{0}^{\delta}\psi(t)dt=1$, and let
	\begin{equation*}
	\varphi(t)=1-c\int_{0}^{t}\int_{0}^{s}\psi(u)duds, t\in[0,\delta].
	\end{equation*}
	Then  $\varphi\in\Phi$ and the solution $y$ to problem \eqref{eq:DDE}, with initial function $\varphi$, is positive.
\end{proposition}
\begin{proof}
	Let $\varphi_1$ be a function on $[0,\delta]$ such that $\varphi_1=-\psi$ on $[0,\delta)$ and $\varphi_1(\delta)=0$. Let $y_1$ be the solution to \eqref{eq:DDE} with initial function $\varphi_1$. Then, from Theorems \ref{thm:fundamental>0} and \ref{thm:compare} we have $y_1(t)>0$, for all $t\ge\delta$. In particular, from \eqref{eq:step1}, we get
	$y_1(t)=-c\int_{0}^{t-\delta}\varphi_1(s)ds>0$, for $t\in[\delta,2\delta]$.
	Moreover, from \eqref{eq:stepk} with $k=2$, we obtain
	\begin{equation*}
	\begin{split}
	y_1(t)&=y_1(2\delta)-c\int_{\delta}^{t-\delta}y_1(s)ds\\
	&=-c\int_{0}^{\delta}\varphi_1(s)ds+c^2\int_{\delta}^{t-\delta}\int_{0}^{s-\delta}\varphi_1(u)duds\\
	&=c+c^2\int_{0}^{t-2\delta}\int_{0}^{s}\varphi_1(u)duds\\
	&>0, \text{ for } t\in[2\delta,3\delta].
	\end{split}
	\end{equation*}
	 Finally, dividing $y_1(t)$ by  $c$ and shifting the function from $[2\delta,3\delta]$ to $[0,\delta]$, we have
	\begin{equation*}
	\frac{y_1(t+2\delta)}{c}=1-c\int_{0}^{t}\int_{0}^{s}\psi(u)duds, t\in[0,\delta],
	\end{equation*}
	It is easy to see that $\varphi\in\Phi$. Indeed, $\varphi(0)=1$ and $\varphi$ is strictly decreasing, with $\varphi(\delta)= y_1(3\delta)/c>0$. To conclude, note that the positivity of $y$ follows from that of $y_1$. See Remark \ref{rem:3} (ii).
\end{proof}
\subsubsection{Properties of solutions $F\in\Pb$}
In the following proposition we present upper and lower bounds of the solutions to problem $\Pb$.
\begin{proposition}
	\label{prop:aprox}
	Let $F\in\Pb$ and let $\varphi$ be the restriction of $1-F$ to the interval $[0,\delta]$. Then, if the conditions of Proposition \ref{prop:suf} are satisfied, the following bounds hold:
	\begin{equation}\label{cotas}
	a_n\le G((n+2)\delta)\le a_1(1-2a)^{n-1},\, n\ge 2,
	\end{equation}
	where  $(a_n)$ solves  \eqref{eq:rec}, with 
	$a_0=G(2\delta)=\varphi(\delta)-2aI_1>0$ and  $a_1=G(3\delta)=(1-2a)\varphi(\delta)-2a(I_1-2aI_2)$.
\end{proposition}
\begin{proof} Let $y$ be the solution to \eqref{eq:DDE}, with initial function $\varphi\in\Phi$, so that, by Theorem \ref{thm:bijection}, $y=G$. The lower bound in \eqref{cotas} is established in the proof of Proposition \ref{prop:suf}. The upper bound comes from the following simple observation:
	\begin{equation*}
		y((n+2)\delta)=y((n+1)\delta)-c\int_{n\delta}^{(n+1)\delta}y(t)dt\le(1-2a)y((n+1)\delta), \, n\ge0.
	\end{equation*}
	Indeed, iterating the recurrence above yields $y((n+2)\delta)\le(1-2a)^{n-1}y(3\delta)$.
\end{proof}
\begin{remark} The bounds \eqref{cotas} allow us to approximate the solution $y$ at points $n\delta$; approximations for intermediate values are direct since $y$ is decreasing. Note that the bounds can be readily computed since we have the explicit form of the terms in the recurrence $(a_n)$, given in \eqref{eq:Hal}. While the exact value of the solution $y$ can be found by the method of steps, the bounds are computed much faster. Moreover, they allow the study of  properties of $y$, which is not possible by using the method of steps, as no analytical expression is obtained.
\end{remark}

It would be interesting to check whether the bounds above are tight. Although we have explicit expressions for the bounds, an analytical comparison seems difficult. For this reason, we present an illustrative example, with $\delta=1$, $c=0.2$, $\varphi(x)=1-x/2$. Figure \ref{figura2} shows the solution computed by the method of steps for $x\in[0,20]$, along with the lower and upper bounds for $y(n)$, $n\ge 4$, from \eqref{cotas}. The values of the bounds for $y(0),y(1),y(2)$ and $y(3)$ are defined as the actual values, which can be written in terms of $\varphi(\delta)$, $I_1$, and $I_2$. The plot shows that the bounds are very close to each other and to the solution; see Table \ref{tabla1} for detailed information.

\begin{figure}[h]
	\centering
	\includegraphics[width=10cm]{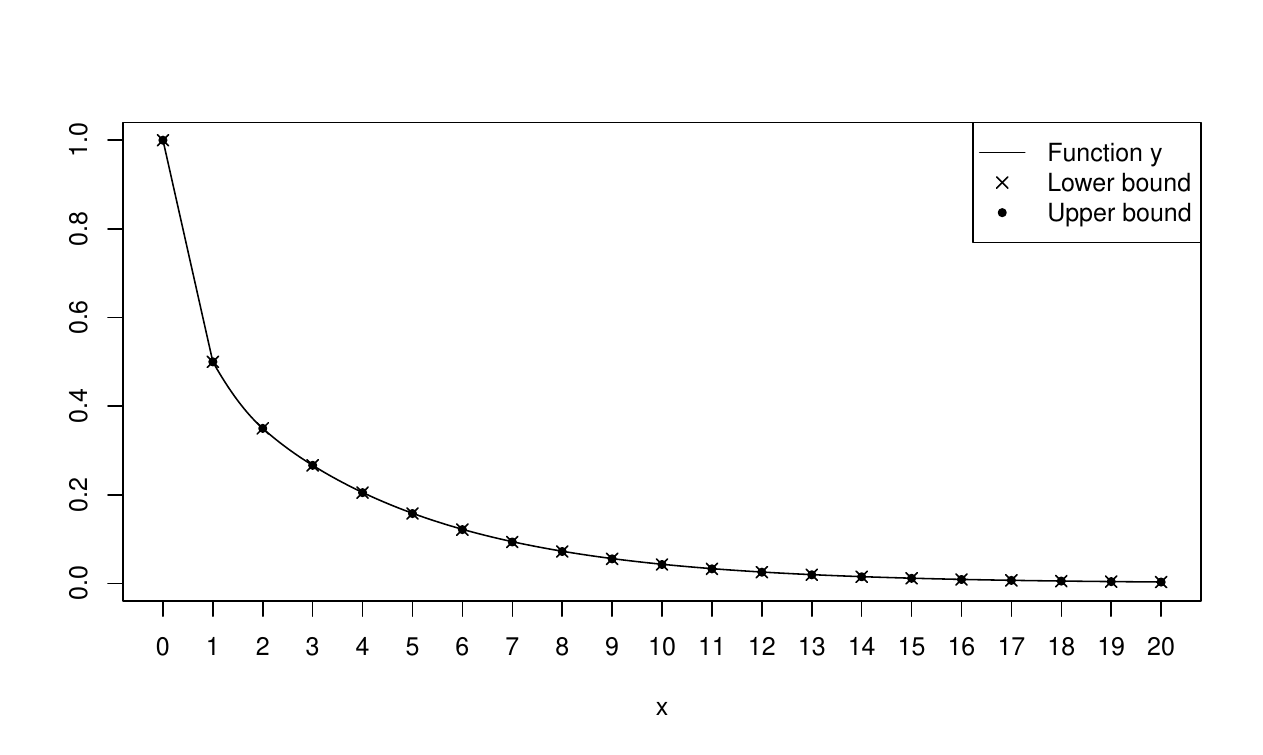}\vspace{-10pt}
	\caption{\footnotesize Function $y$ with lower and upper bounds for $c=0.2$, $\delta=1$ and $\varphi(x)=1-x/2$.}
	\label{figura2}
\end{figure}
\begin{table}[h!]\footnotesize
	\centering
	\begin{tabular}{|c|c|c|c|}
		\hline
		$x$ & Lower bound & $y(x)$ & Upper bound \\
		\hline
		0  & 1.0000 & 1.0000 & 1.0000 \\
		1  & 0.5000 & 0.5000 & 0.5000 \\
		2  & 0.3498 & 0.3498 & 0.3498 \\
		3  & 0.2665 & 0.2665 & 0.2665 \\
		4  & 0.2048 & 0.2053 & 0.2132 \\
		5  & 0.1577 & 0.1583 & 0.1705 \\
		6  & 0.1215 & 0.1222 & 0.1364 \\
		7  & 0.0935 & 0.0942 & 0.1091 \\
		8  & 0.0720 & 0.0727 & 0.0873 \\
		9  & 0.0555 & 0.0561 & 0.0699 \\
		10 & 0.0427 & 0.0433 & 0.0559 \\
		11 & 0.0329 & 0.0334 & 0.0447 \\
		12 & 0.0253 & 0.0257 & 0.0358 \\
		13 & 0.0195 & 0.0199 & 0.0286 \\
		14 & 0.0150 & 0.0153 & 0.0229 \\
		15 & 0.0116 & 0.0118 & 0.0183 \\
		16 & 0.0089 & 0.0091 & 0.0146 \\
		17 & 0.0069 & 0.0070 & 0.0117 \\
		18 & 0.0053 & 0.0054 & 0.0094 \\
		19 & 0.0041 & 0.0042 & 0.0075 \\
		20 & 0.0031 & 0.0032 & 0.0060 \\
		\hline
	\end{tabular}\vspace{5pt}
	\caption{\footnotesize Function $y$ with lower and upper bounds for $c=0.2$, $\delta=1$ and $\varphi(x)=1-x/2$.}
	\label{tabla1}
\end{table}
In the following result we give an explicit formula for the Laplace transform of $F\in\Pb$ and a recurrence for its moments.
\begin{proposition}
	Let $F\in\Pb$ and let $X$ be a random variable with distribution $F$. Then 
	\begin{enumerate}
		\item $\mu_n:=E(X^n)<\infty$, for all $n\ge1$ and the following recurrence holds
		\begin{equation}
		\label{eq:recmom}		\mu_{n+1}=\frac{n+1}{c}\left((1-c\delta)\mu_n-\int_{0}^{\delta}t^nF(dt)\right)-\sum_{k=1}^{n-1}\binom{n+1}{k}\mu_k\delta^{n+1-k}.
		\end{equation}
		\item \begin{equation}
		\label{eq:genfun}
		E(e^{-uX})=\frac{\int_{0}^{\delta}e^{-ut}F(dt)
			+ce^{-u\delta}/u}{1+ce^{-u\delta}/u},\, u>0.
		\end{equation}
	\end{enumerate}
\end{proposition}
\begin{proof}
	(a) Let $y$ be the solution to \eqref{eq:DDE}, with initial function $1-F$ on $[0,\delta]$. Then
	\begin{equation*}
	\begin{split}
	E(X^n)&=\int_{0}^{\infty}t^nF(dt)\\
	&=L_n-\int_{\delta}^{\infty}t^ny'(t)dt\\
	&=L_n+c\int_{0}^{\infty}(u+\delta)^ny(u)du\\
	&=L_n-c\int_{0}^{\infty}y'(x)\int_{0}^{x}(u+\delta)^ndudx\\
	&=L_n+\frac{c}{n+1}\left(E(X+\delta)^{n+1}-\delta^{n+1}\right),
	\end{split}
	\end{equation*}
	where $L_n=\int_{0}^{\delta}t^nF(dt)$. Recurrence \eqref{eq:recmom} is obtained by solving for $\mu_{n+1}$.
	
	(b) Formula \eqref{eq:genfun} follows from \eqref{eq:laplacey} in Lemma \ref{lem:laplacey}.
\end{proof}

\begin{remark}
A well-known result of S. Bernstein, related to Laplace transforms, states that a function $f:(0,\infty)\to\RR$ is the Laplace transform of a probability measure if and only if it is completely monotonic (c.m.) and $\lim_{t\to0^+}f(t)=1$;  see p. 417 in \cite{Fel}. Recall that $f$ is said to be c.m. if it has derivatives of all orders and $(-1)^nf^{(n)}(x)\ge0$, for all $n\ge0$ and $x>0$. Therefore, there is a bijection  between continuous $F\in\Pb$, with $\suppF=\RR_+$, and functions $\varphi\in\Phi$ such that \eqref{eq:laplacey} is c.m. This provides a way of checking whether an initial function $\varphi\in\Phi$ generates a solution to $\Pb$.
However, this criterion is hardly applicable in practice, as higher-order derivatives of \eqref{eq:laplacey} become unmanageable. 

On the other hand, given that we have the simple sufficient conditions \eqref{eq:suf} and \eqref{eq:suf2} for the positivity of the solution $y$ to \eqref{eq:DDE}, we can produce a range of examples of c.m. functions. For the sake of illustration, take  $\delta=1, c=1/5$ and $\varphi(t)=1-t/5$ for $t\in[0,\delta]$. Clearly, $\varphi\in\Phi$, and \eqref{eq:suf2} holds. Thus, from \eqref{eq:laplacey}, it follows that $\hat{y}(u)=(5u+e^{-u})^{-1}$ is c.m. However, $\hat{y}(u)=(2u+e^{-u})^{-1}$ is not c.m. because it corresponds to the case with $\delta=1, c=1/2$ and $\varphi(t)=1-t/2$ for $t\in[0,\delta]$, which has $c\delta> 1/e$ and so, $\Pb=\emptyset$, by Theorem \ref{thm:bijection}.
\end{remark}

\subsection{Lattice distributions}
This last section is devoted to the analysis of lattice solutions to problem $\Pb$, with $\delta>0$. For simplicity, we consider distributions with support $\ZZ_+$ and, without loss of generality, we take $\delta\in\ZZ_+$.

When $\suppF=\ZZ_+$, we have $T=\suppF$ (recall Definition \ref{def:ST}). Lemma \ref{L1} states that $F\in \Pb$ if and only if  $G(i+\delta)=c\int_{i}^{\infty}G(t)dt$, for all $i\in\ZZ_+$ or, equivalently, if $G(i)=c\int_{i-\delta}^{\infty}G(t)dt=c\sum_{j=i-\delta}^{\infty}G(j)$, for all $i\in\ZZ_+,i\ge\delta$. Thus, $G$ satisfies the difference equation $\Delta G(i)+cG(i-\delta)=0, i\ge\delta$, with initial condition $G(0)<1$, where $\Delta G(i):=G(i+1)-G(i)$ denotes the forward difference operator on $G$.

For any $F\in\Pb$, the values of $G$ on $\ZZ_+$ are determined by the values of $G$ on $\{0,\ldots,\delta\}$. That is, the solutions to $\Pb$ are parametrised by decreasing initial functions $\varphi$ defined on the discrete set $\{0, \ldots, \delta\}$, analogous to the initial functions defined on the interval $[0, \delta]$ in the continuous case. As we will see, there is a strong similarity between the results in the continuous and discrete settings.

Let the discrete initial value problem be defined by
\begin{equation}
	\label{eq:DdE}
	\Delta y(i)+cy(i-\delta)=0, \, i\ge\delta \quad\text{ and }\quad y(i)=\varphi(i), \, i=0,\ldots,\delta.
\end{equation}
Note that, if $y$ satisfies \eqref{eq:DdE}, then
\begin{equation}
	\label{eq:DdE1}
	y(i)=y(l)-c\sum_{j=l-\delta}^{i-\delta-1}y(j), \, i,l\in\ZZ_+ \text{ such that } i>l\ge\delta,
\end{equation}
which is the discrete analogue of \eqref{eq:DDE1}.

The (unique) solution to \eqref{eq:DdE} on $\{\delta+1,\ldots\}$ is computed as the extension of $\varphi$ using the discrete version of the method of steps, based on \eqref{eq:DdE1}. For example, the first step yields 
	\begin{equation}
	\label{eq:dG1}
	y(i)=\varphi(\delta)-c\sum_{j=0}^{i-\delta-1}\varphi(j), \,i=\delta+1,\ldots,2\delta.
\end{equation}
In general, for $k\ge1$,
\begin{equation}
	\label{eq:stepkd}
	y(i)=y(k\delta)-c\sum_{j=(k-1)\delta}^{i-\delta-1}y(j), \,i=k\delta+1, \ldots,(k+1)\delta.
\end{equation}
Since we are interested in solutions to \eqref{eq:DdE} that are survival functions, $\varphi$ must take values in $(0,1)$ and be strictly decreasing. The set of such functions is denoted by $\Phi_d$. In the discrete setting, the result similar to Theorem \ref {thm:bijection} can be stated as follows: 
\begin{theorem} 
	\label{thm:bijectiondis}
There exists a bijection between the set of  distributions $F\in\Pb$, with $\suppF=\ZZ_+$, and the set of positive solutions $y$ to \eqref{eq:DdE}, with  initial function $\varphi\in\Phi_d$. The bijection is given by $F=1-y$. In particular, $\Pb=\emptyset$  if and only if $c\delta>\left(\frac{\delta}{\delta+1}\right)^{\delta+1}$.
\end{theorem}
\begin{proof}
Clearly, every $F\in\Pb$ with $\suppF=\ZZ_+$ is a solution to \eqref{eq:DdE}, with initial function $\varphi\in\Phi_d$ given by $\varphi(i)=1-F(i)$ for $i=0,\ldots,\delta$. Conversely, any positive solution $y$ to \eqref{eq:DdE}, with initial function $\varphi \in \Phi_d$, is strictly decreasing and, by Lemma \ref{lem:dexpdecay}, satisfies $y(i) \to 0$ as $i \to \infty$. Hence, $1-y$ is a distribution function in $\Pb$. Finally, by Theorem \ref{thm:oscdis}, if $c\delta>\left(\frac{\delta}{\delta+1}\right)^{\delta+1}$ there are no positive solutions to \eqref{eq:DdE}, so $\Pb=\emptyset$.
\end{proof}
We study necessary and sufficient conditions on the initial function $\varphi\in\Phi_d$ for the positivity of solutions to \eqref{eq:DdE}. In light of Theorem \ref{thm:bijectiondis}, we shall henceforth assume that $c\delta \le \left(\frac{\delta}{\delta+1}\right)^{\delta+1}$.

\subsubsection{Necessary condition}
\begin{proposition}
	\label{prop:dnec}
	If $y$ is a positive solution to problem \eqref{eq:DdE}, with initial function $\varphi\in\Phi_d$, then
\begin{equation}
\label{eq:dnec2}
\varphi(\delta)>\frac{c\delta}{1-c\delta}(S_1-c\delta S_2),
\end{equation}
where
	$S_1=\frac{1}{\delta}\sum_{j=0}^{\delta-1}\varphi(j)$ and $S_2=\frac{1}{\delta^2}\sum_{j=\delta}^{2\delta-1}\sum_{i=0}^{j-\delta-1}\varphi(i)$.
\end{proposition}
\begin{proof}
Note that the rhs of \eqref{eq:dnec2} is positive   since $c\delta<1$ and $S_2\le S_1$. From \eqref{eq:dG1}, with $i=2\delta$, we get $y(2\delta)=\varphi(\delta)-c\delta S_1$.

Now, from \eqref{eq:stepkd}, with $k=2$ and $i=3\delta$, we find that
 $y(3\delta)>0$ is equivalent to $y(2\delta)> c\sum_{j=\delta}^{2\delta-1}y(j)$, which, in turn, is equivalent to
\begin{equation}
	\label{eq:dnec1.5}
	\begin{split}
		\varphi(\delta)- c\delta S_1&> c\sum_{j=\delta}^{2\delta-1}\left(\varphi(\delta)-c\sum_{i=0}^{j-\delta-1}\varphi(i)\right)
		\\
		&=c\delta\varphi(\delta)-c^2\sum_{j=\delta}^{2\delta-1}\sum_{i=0}^{j-\delta-1}\varphi(i)
		\\ &=c\delta\left(\varphi(\delta)-c\delta S_2\right).
		\end{split}
\end{equation}
Solving for $\varphi(\delta)$ in \eqref{eq:dnec1.5} and recalling that $c\delta<1$, we obtain \eqref{eq:dnec2}.
\end{proof}
\subsubsection{Sufficient condition}
 As in the continuous case, we derive sufficient conditions on $\varphi\in\Phi_d$ to ensure a positive solution. Again, the idea is to find a positive sequence $(a_n)$ that serves as a lower bound for $y$.
 We begin with  a technical lemma, analogous to Lemma \ref{lem:bn1}, and redefine parameter $a$ as $a=c(\delta+1)/2<1$.

 \begin{lemma}
 	\label{lem:dbn1}
 	Let $y$ be the solution to \eqref{eq:DdE}, with initial function $\varphi\in\Phi_d$. Define  $b_k=y((k+1)\delta)$, for $k\ge0$, and let $n\ge1$. If $b_k>0$, for $k=1,\ldots,n$, then $y$ is positive and decreasing, while $\Delta y$ is increasing on $J_k:=\{k\delta,\ldots,(k+1)\delta\}$, for $k=1,\ldots,n$. Moreover, 
 	\begin{equation*}
 		b_{k+1}\ge (1-a)b_{k}-ab_{k-1},
 	\end{equation*}
 	for $k=1,\ldots,n$.
 \end{lemma}
\begin{proof}
	The proof closely parallels that of Lemma \ref{lem:bn1} with certain details omitted for brevity. Let $p(n)=``b_k>0$, for $k=1,\ldots,n$'';  $q(n)=``y$ is positive and decreasing, while $\Delta y$ is increasing on $J_k$, for $k=1,\ldots,n$'' and $r(n)=``b_{k+1}\ge (1-a)b_{k}-ab_{k-1}$, for $k=1,\ldots,n$''. We  prove that $p(n)\impl q(n)$ and $q(n)\impl r(n)$, for all $n\ge1$.
	
	In the initial step of the induction assume, $b_1>0$ and observe from \eqref{eq:dG1} that $y$ decreases and $\Delta y$ increases on $J_1$, because $\varphi$ is positive and decreases on $\{0,\ldots,\delta\}$; so  $p(1)\impl q(1)$ holds. Now, suppose $p(n)\impl q(n)$ and assume that $p(n+1)$ holds. As $p(n+1)$ implies $p(n)$, we have $q(n)$ which, by \eqref{eq:stepkd} with $k=n+1$, yields $q(n+1)$ and the induction is complete.
	
	For the second implication, note that $y$ satisfies the hypotheses of Lemma \ref{lem:dconvex} on $J_k$ and hence the following bound holds for $k=1,\ldots,n$:
	\begin{equation*}
		\sum_{j=k\delta}^{(k+1)\delta-1}y(j)\le\sum_{j=k\delta}^{(k+1)\delta}y(j)\le \tfrac{\delta+1}{2}(y(k\delta)+y((k+1)\delta)).
	\end{equation*}
	The inequality above and \eqref{eq:stepkd}, with $k$ replaced by $k+1$, yield
	\begin{equation*}
		b_{k+1}=y((k+1)\delta)-c\!\!\!\!\sum_{j=k\delta}^{(k+1)\delta-1}y(j)\ge (1-a)y((k+1)\delta)-ay(k\delta)\!=\!(1-a)b_k-ab_{k-1},
	\end{equation*}
	for $k=1,\ldots,n$, and the proof is complete.
\end{proof}
\begin{proposition}
	\label{prop:dsuf} Suppose $a<3-2\sqrt{2}$, and let $y$ be the solution to \eqref{eq:DdE} with initial function $\varphi\in\Phi_d$. Then $y$ is positive if 
	\begin{equation*}
	\varphi(\delta)>\frac{2a}{1-\lambda_2-2a}((1-\lambda_2)S_1-2aS_2),
\end{equation*} 
where $\lambda_2=\tfrac{1}{2}(1-a-\sqrt{D})$, $D=(1-a)^2-4a$, and $S_1, S_2$ are defined in Proposition \ref{prop:dnec}.
\end{proposition}
\begin{proof}
	The proof is similar to that of Proposition \ref{prop:suf}, with the difference that now $a=c(\delta+1)/2$ instead of $c\delta/2$, and $b_1=y(2\delta)=\varphi(\delta)-c\delta S_1$, $b_2=y(3\delta)=(1-c\delta)\varphi(\delta)-c\delta(S_1-c\delta S_2)$. Thus, it suffices to prove the inequalities $b_1>0$ and $b_2>b_1\lambda_2$, which are satisfied under the hypotheses of the proposition.
\end{proof}

\begin{example}\label{geometrica}(Geometric distribution)
Let $F$ be the geometric distribution starting at 0 with parameter $p\in(0,1)$ (denoted $Geom(p)$). That is, $G(t)=1-F(t)=(1-p)^{\lfloor t\rfloor+1}, t\ge0$. From Lemma \ref{L1} (i), $F\in\Pb$ if and only if
	\begin{equation*}
	(1-p)^{k+\delta+1}=c\int_{k}^{\infty}G(t)dt=c\sum_{i=k}^{\infty}(1-p)^{i+1}=c\frac{(1-p)^{k+1}}{p}, k\ge0,
	\end{equation*}
	which is equivalent to $p(1-p)^{\delta}=c$. Therefore, the geometric distribution starting at 0, with parameter $p$, solves $\Pb$ if and only if $p$ is a solution to the equation $p(1-p)^{\delta}=c$. It is easy to check that a solution exists if and only if $c\delta\le \left(\frac{\delta}{\delta+1}\right)^{\delta+1}$. When $c\delta< \left(\frac{\delta}{\delta+1}\right)^{\delta+1}$, there are two solutions, $p_1$ and $p_2$, say. Thus $Geom(p_1)$ and $Geom(p_2)$ are solutions to $\Pb$. Moreover, given that \eqref{eq:DdE} is linear, the distributions with $G(t)=b_i(1-p_i)^{\lfloor t\rfloor+1}$, for $t\ge0$, where $b_i\in(0,1/(1-p_i))$, $i=1,2$, are also in $\Pb$. Thus, from the previous results, we find that all mixtures of a Dirac mass at 0 and the $Geom(p_i)$ distributions, $i = 1,2$, are solutions to $\Pb$. 
	
	When $c\delta= \left(\frac{\delta}{\delta+1}\right)^{\delta+1}$, the only real solution is $\frac{1}{\delta+1}$, and so, $Geom(\frac{1}{\delta+1})$ is a solution, as well as its mixture with the Dirac mass at 0. Moreover, there are solutions of the form  $G(t)=(\alpha\lfloor t\rfloor+1) \left(\frac{\delta}{\delta+1}\right)^{\lfloor t\rfloor+1}$, for $t\ge0$, with $\alpha\in(0, 1/\delta)$. Summarising, the mixtures of a Dirac mass at 0, a $Geom(\frac{1}{\delta+1})$, and a negative binomial distribution with parameters $2,\,(\delta+1)^{-1}$ are solutions to $\Pb$.
\end{example}

\begin{appendices}

\section{Technical results}
\label{sec:appendix}
In this appendix we collect definitions and technical results about delay differential equations and difference equations.
\subsection{Delay differential equations}
We define the fundamental function of \eqref{eq:DDE}. See \cite{AGA}, pages 3--5.
\begin{definition}
	\label{def:fundamental}
	The fundamental function of \eqref{eq:DDE}, denoted $y_0$, is defined as the solution to \eqref{eq:DDE}, with initial function $\varphi_0(t)=0$ for $t\in[0,\delta)$, and $\varphi_0(\delta)=1$.
\end{definition}
The following important result is used in Proposition \ref{prop:initial}. See \cite{AGA} for a proof.
\begin{theorem}
	\label{thm:fundamental>0}
	If $c\delta\le1/e$, then the fundamental function $y_0$ is positive.
\end{theorem}
The analytical properties of the solutions of DDEs such as \eqref{eq:DDE} have been extensively studied in the literature. Here, we restate two important results concerning the positivity and subexponentiality of the solutions. See Definition 1.2.2, Theorem 2.1.2, and Corollary 2.1.1 in \cite{LLZ}.
\begin{theorem}
	\label{thm:deltac<1/e}
	Problem \eqref{eq:DDE} has a positive solution if and only if $c\delta\le 1/e$.
\end{theorem}
\begin{lemma}
	\label{lem:expdecay}
	Let $y$ be a solution of \eqref{eq:DDE} with continuous initial function $\varphi$. 
	
	(i) If $c\delta<\pi/2$, then there exist positive constants $M$ and $\nu$ such that
	\begin{equation}
	\label{eq:expdecay}
	|y(t)|\le Me^{-\nu(t-\delta)}, \,  t\ge\delta.
	\end{equation}
	Moreover, 
	
	(ii) if $y$ is a positive solution of \eqref{eq:DDE}, then $y$ is decreasing and
	\begin{equation}
		\label{eq:Gron}
		 y(t)\le y(\delta)e^{-c(t-\delta)},\, t\ge\delta.
	\end{equation}
\end{lemma}
\begin{proof}
	For a proof of \eqref{eq:expdecay}, see  Lemma 2.1.1 in \cite{LLZ}. For \eqref{eq:Gron}, note that $y$ is decreasing on $(\delta, \infty)$ since $y(t-\delta)>0$. Now, as $y(t-\delta)\ge y(t)$, we 	
	 obtain $y'(t)\le -cy(t),t\ge\delta$, which yields \eqref{eq:Gron} by Gr{\"o}nwall's lemma (see \cite{GRO}, p. 293).
\end{proof}
The following is a general comparison theorem for DDEs, not requiring $\varphi_1, \varphi_2\in\Phi$. The initial functions can be, for example, non-decreasing or negative.

\begin{theorem}
	\label{thm:compare}
	Let $y_1, y_2$ be the solutions to problem \eqref{eq:DDE} with respective initial functions $\varphi_1, \varphi_2$, such that $y_2(t)>0$ for all $t>\delta$. If $\varphi_1(t)\le\varphi_2(t)$, for all $t\in[0,\delta]$ and $\varphi_1(\delta)= \varphi_2(\delta)$, then $y_1(t)\ge y_2(t)>0$, for all $t\ge\delta$.
\end{theorem}
\begin{proof}
	The result is a particular case of Theorem 2.5 in \cite{AGA}.
\end{proof}
\subsubsection{The Laplace-Stieltjes transform}
The Laplace-Stieltjes transform is a classical tool in the analysis of linear DDE. The simplicity of \eqref{eq:DDE} allows for an explicit formula in terms of the initial function $\varphi$. This is, of course, well-known; see \cite{BC}.
\begin{definition}
	\label{def:laplace}
	Let $y$ be a solution to \eqref{eq:DDE} with initial function $\varphi$ of bounded variation. The Laplace-Stieltjes transform of $y$ is defined as
	\begin{equation*}
	\hat{y}(u)=-\int_0^\infty e^{-ut}y(dt), u>0.
	\end{equation*}
\end{definition}
If $\varphi\in\Phi$ and $c\delta\le 1/e$, as in Section \ref{sec:continuous}, then $\varphi$ is of bounded variation and, by  Lemma \ref{lem:expdecay} (i), the integral defining $\hat{y}(u)$ converges  for all $u>0$. In the next lemma, we present a compact formula for $\hat{y}$.
\begin{lemma}
	\label{lem:laplacey}
	Using the notation and conditions of Definition \ref{def:laplace}, 
	\begin{equation}
	\label{eq:laplacey}
	\hat{y}(u)=\frac{L_\varphi(u)+ce^{-u\delta}/u}{1+ce^{-u\delta}/u},\, u>0,
	\end{equation}
	where $L_\varphi(u)=-\int_{0}^{\delta}e^{-ut}\varphi(dt)$.
\end{lemma}
\begin{proof}
	From  \eqref{eq:DDE} we have
	\begin{equation*}
	\begin{split}
	\hat{y}(u)&=L_\varphi(u)+c\int_{\delta}^{\infty}e^{-ut}y(t-\delta)dt\\
	&=L_\varphi(u)+ce^{-u\delta}\int_{0}^{\infty}e^{-us}y(s)ds\\
	&=L_\varphi(u)+ce^{-u\delta}\int_{0}^{\infty}e^{-us}\left[\int_{0}^{s}y'(t)dt+y(0)\right]ds\\
	&=L_\varphi(u)+\tfrac{ce^{-u\delta}}{u}(1-\hat{y}(u)).
	\end{split}
	\end{equation*}
	Then, solving for $\hat{y}$, we  get \eqref{eq:laplacey}.
\end{proof}
\subsection{Difference equations}
We present definitions and results for difference equations that closely parallel those for DDEs.
\begin{theorem}
\label{thm:oscdis}
	Problem \eqref{eq:DdE} has a positive solution if and only if 
	$c\delta\le \left(\frac{\delta}{\delta+1}\right)^{\delta+1}$.
\end{theorem}
\begin{proof}
	See Theorems 2.1--2.3 in \cite{EZ}.
\end{proof}
\begin{lemma}
	\label{lem:dexpdecay}
	If $y$ is a positive solution of  \eqref{eq:DdE}, then $y$ is decreasing and
	\begin{equation}
		\label{eq:dGron}
		y(k)\le y(\delta)(1-c)^{k-\delta}, \, k\ge\delta.
	\end{equation}
\end{lemma}
\begin{proof}
	For \eqref{eq:dGron}, note that $y$ is decreasing on $\{\delta,\delta+1,\ldots\}$, since $y(k-\delta)>0$ for $k\ge\delta$. Also, since $y(k-\delta)\ge y(k)$, we get $\Delta y(k)\le -cy(k)$ for $k\ge\delta$, which yields \eqref{eq:dGron}. Finally, by Theorem \ref{thm:oscdis}, the existence of a positive solution to \eqref{eq:DdE} implies $c\delta\le \left(\frac{\delta}{\delta+1}\right)^{\delta+1}$.
\end{proof}

The following result states conditions for the positivity of the solution to a recurrence. It is mainly used in the proof of Proposition \ref{prop:suf}.
\begin{lemma} 
	\label{lem:rec}
	Let $a, a_0, a_1$ be positive constants such that $a<3-2\sqrt{2}$ and $a_1>a_0\lambda_2$, where $\lambda_2=\tfrac{1}{2}(1-a-\sqrt{D})$ and $D=(1-a)^2-4a$. Then the recurrence
	\begin{equation}
		\label{eq:rec}
		a_n=(1-a)\,a_{n-1}-a\,a_{n-2},\, n\ge2,
	\end{equation}
	has a positive solution $(a_n)$.
\end{lemma}
\begin{proof}
	The characteristic polynomial of \eqref{eq:rec} is $p(x)=x^2-(1-a) x+a$, and has roots 
	\begin{equation*}
		\lambda_1=\tfrac{1}{2}(1-a+\sqrt{D}),\quad \lambda_2=\tfrac{1}{2}(1-a-\sqrt{D}).
	\end{equation*}
	Clearly, $a<3-2\sqrt{2}$ implies that $\lambda_1, \lambda_2$ are real and satisfy $0<\lambda_2<\lambda_1<1$. Then (see Lemma 1 in \cite{HHH}),
	\begin{equation}
		\label{eq:Hal}
		a_n=A\lambda_1^n+B\lambda_2^n, \, n\ge0,
	\end{equation}
	where $A=\frac{a_1-a_0\lambda_2}{\sqrt{D}}$ and $B=\frac{a_0\lambda_1-a_1}{\sqrt{D}}$. 
	Our assumption $a_1>a_0\lambda_2$ implies $A>0$. Now, if $a_0\lambda_1-a_1\ge0$, then $B\ge0$, and we have $a_n>0$ for all $n\ge0$. Otherwise, if $B<0$, then since $\lambda_2/\lambda_1<1$, the minimum value of $A+B\left(\frac{\lambda_2}{\lambda_1}\right)^n$ is achieved at $n = 0$. At this point, it equals $A + B = a_0 > 0$, thus proving the claim.
\end{proof}

\begin{lemma}
	\label{lem:bn2}
	Let $(a_n)$ be the solution to recurrence \eqref{eq:rec} under the hypotheses of Lemma \ref{lem:rec}. Let $n\ge2$, $x_0=a_0, x_1=a_1$, and $x_2,\ldots,x_n\in\RR$  such that 
	\begin{equation}
		\label{eq:recb}
		x_k\ge (1-a)x_{k-1}-ax_{k-2},\, k=2,\ldots,n.
	\end{equation}
	Then $x_k\ge a_k, \text{ for } k=0,\ldots,n$.
\end{lemma}
\begin{proof}
We argue by induction. The statement is obviously true for $n=0,1$, so let us assume that \eqref{eq:recb} and  $x_k\ge a_k$ hold for $k=2,\ldots,n$. If  $x_{n+1}\ge (1-a)x_{n}-ax_{n-1}$, then
	\begin{equation}
		\label{eq:disentangle}
		\begin{split}
		x_{n+1}&\ge (1-a)((1-a)x_{n-1}-ax_{n-2})-ax_{n-1}\\
					&=((1-a)^2-a)x_{n-1}-a(1-a)x_{n-2}\\
					&\ge (1-a)(((1-a)^2-a)-a)x_{n-2}-a((1-a)^2-a)x_{n-3}\\
					&\vdots\\
					&\ge\alpha_kx_{n-k+1}-a\alpha_{k-1}x_{n-k},
		\end{split}
	\end{equation}
	where $\alpha_0=1$, $\alpha_1=1-a$, and  
	\begin{equation}
		\label{eq:recalpha}
		\alpha_k=(1-a)\alpha_{k-1}-a\alpha_{k-2},\, k=2,\ldots,n.
	\end{equation}
	The iterative procedure in display \eqref{eq:disentangle}, used to disentangle the recurrent inequalities, is justified only if the coefficients $\alpha_k$ are positive. 

	Observe that the recurrence \eqref{eq:recalpha} defining the sequence $(\alpha_k)$ for $k \geq 2$ is identical to \eqref{eq:rec}. Thus, we can apply Lemma \ref{lem:rec} to this sequence with $\alpha_0=1$ and $\alpha_1=1-a$. Noting that $\alpha_0>0$ and $\alpha_1=1-a>(1-a-\sqrt{D})/2=\alpha_0\lambda_2$, we conclude that $\alpha_k>0$ for all $k\ge0$.
Therefore, the iterative scheme in \eqref{eq:disentangle} is justified. 
Note also, from \eqref{eq:Hal} applied to the sequence $(\alpha_k)$, that
\begin{equation}
\label{eq:an}
\alpha_k=\frac{1-a-\lambda_2}{\sqrt{D}}\lambda_1^k+\frac{\lambda_1-1+a}{\sqrt{D}}\lambda_2^k
= \frac{\lambda_1^{k+1}-\lambda_2^{k+1}}{\sqrt{D}}, \, k\ge2.
\end{equation}
Finally, taking $k=n$ and noting that $\lambda_1\lambda_2=a$, from \eqref{eq:disentangle} and \eqref{eq:an}, we obtain
		\begin{equation*}
		\begin{split}
			x_{n+1}&\ge \alpha_nx_1-a\alpha_{n-1}x_0\\
			&=\alpha_na_1-a\alpha_{n-1}a_0\\
			&=\frac{\lambda_1^{n+1}-\lambda_2^{n+1}}{\sqrt{D}}a_1-\lambda_1\lambda_2\frac{\lambda_1^{n}-\lambda_2^{n}}{\sqrt{D}}a_0\\
			&=a_{n+1}.
		\end{split}
	\end{equation*}
	Therefore, the proof of the inductive step is complete.
\end{proof}
\begin{lemma}
	\label{lem:dconvex} Let $m,n\in\ZZ_+, n\ge1$, and $A=\{m,\ldots,m+n\}$. Let $g:\ZZ_+\to\RR_+$ be a decreasing and discrete-convex function on $A$, in the sense that $g(m)>\cdots>g(m+n)$ and $\Delta g(m)\le\cdots\le\Delta g(m+n-1)$. Then
	\begin{equation}
		\label{eq:dconvex}
		\sum_{i=m}^{m+n}g(i)\le\tfrac{n+1}{2}(g(m)+g(m+n)).
	\end{equation}
\end{lemma}
\begin{proof} Let $h$ be the function on $[m,m+n]$ defined by the straight line joining the points with coordinates $(m,g(m))$ and $(m+n,g(m+n))$. That is, $h(x)=g(m)-\frac{1}{n}(g(m)-g(m+n))(x-m), x\in[m,m+n]$. Since $g$ is discrete-convex on $A$, we have $g(i)\le h(i)$ for all $i\in A$, and 
		\begin{equation*}
	\sum_{i=m}^{m+n}g(i)\le\sum_{i=m}^{m+n}h(i)=(n+1)g(m)-\tfrac{1}{n}(g(m)-g(m+n))\sum_{k=0}^{n}k,
	\end{equation*}
	which yields \eqref{eq:dconvex}.
\end{proof}

\end{appendices}

\section*{Declarations}

\textbf{Conflict of interest} The authors declare no conflict of interest.

\noindent\textbf{Funding} This research was partially funded by grants PID2023-150234NB-I00 and TED2021-130702BI00 from the Ministry of Science and Innovation. The authors are members of the research group \textit{E46\_23R: Modelos Estoc\'asticos} of DGA (Arag\'on, Spain).

\end{document}